\documentclass[11pt, a4paper]{amsart}

\usepackage{aliascnt}
\usepackage{hyperref}
\usepackage{amsfonts}
\usepackage{mathrsfs}
\usepackage{amsmath}
\usepackage{amsmath,graphics}
\usepackage{amssymb}
\usepackage{indentfirst,latexsym,bm,amsmath,pstricks,amssymb,amsthm,graphicx,fancyhdr,float,color}
\usepackage[all]{xy}
\usepackage{amsthm}
\usepackage{amscd}
\usepackage[all]{hypcap}

\newtheorem{theorem}{Theorem}[section]

\newaliascnt{lemma}{theorem}
\newtheorem{lemma}[lemma]{Lemma}
\aliascntresetthe{lemma}

\newaliascnt{claim}{theorem}
\newtheorem{claim}[claim]{Claim}
\aliascntresetthe{claim}

\newaliascnt{conjecture}{theorem}
\newtheorem{conjecture}[conjecture]{Conjecture}
\aliascntresetthe{conjecture}

\newaliascnt{proposition}{theorem}
\newtheorem{proposition}[proposition]{Proposition}
\aliascntresetthe{proposition}

\newaliascnt{corollary}{theorem}
\newtheorem{corollary}[corollary]{Corollary}
\aliascntresetthe{corollary}

\newaliascnt{problem}{theorem}

\aliascntresetthe{problem}

\theoremstyle{definition}

\newaliascnt{definition}{theorem}
\newtheorem{definition}[definition]{Definition}
\aliascntresetthe{definition}

\newaliascnt{example}{theorem}
\newtheorem{example}[example]{Example}
\aliascntresetthe{example}

\theoremstyle{remark}

\newaliascnt{remark}{theorem}
\newtheorem{remark}[remark]{Remark}
\aliascntresetthe{remark}

\newaliascnt{remarks}{theorem}
\newtheorem{remarks}[remarks]{Remarks}
\aliascntresetthe{remarks}

\numberwithin{equation}{section}

\def\wt{\widetilde}

\def\lra{\longrightarrow}

\def\({$($}
\def\){$)$}
\def\chit{\chi_{\rm top}}

\def\call{\mathcal L}

\def\calo{\mathcal O}

\def\Pic{\text{{\rm Pic\,}}}
\def\rank{\text{{\rm rank\,}}}

\def\oly{\overline{Y}}

\def\im{{\rm Im}}
\def\Alb{{\rm Alb}}

\begin{document}

\title[On the slope conjecture]{On the slope conjecture of Barja and Stoppino for fibred surfaces}

\author{Xin Lu}
\address{Department of Mathematics, East China Normal University, Shanghai, China, 200241}
\curraddr{Institut f\"ur Mathematik, Universit\"at Mainz, Mainz, Germany, 55099}
\email{x.lu@uni-mainz.de}
\thanks{This work is supported by SFB/Transregio 45 Periods, Moduli Spaces and Arithmetic of Algebraic Varieties of the DFG (Deutsche Forschungsgemeinschaft),
and partially supported by National Key Basic Research Program of China (Grant No. 2013CB834202).}

\author{Kang Zuo}
\address{Institut f\"ur Mathematik, Universit\"at Mainz, Mainz, Germany, 55099}
\email{zuok@uni-mainz.de}





\phantomsection
\begin{abstract}
\addcontentsline{toc}{section}{Abstract}
Let $f:\,X \to B$ be a locally non-trivial relatively minimal fibration of genus $g\geq 2$ with relative irregularity $q_f$.
It was conjectured by Barja and Stoppino that the slope $\lambda_f\geq \frac{4(g-1)}{g-q_f}$.
On the one hand, we show the lower bound $\lambda_f> \frac{4(g-1)}{g-q_f/2}$,
and also prove Barja-Stoppino's conjecture when $q_f$ is small with respect to $g$.
On the other hand, we construct counterexamples violating the conjectured bound when $g$ is odd and $q_f=(g+1)/2$.
\end{abstract}

\maketitle

\section{Introduction}
A fibred surface, or simply a fibration, is a surjective proper morphism $f:X \to B$
from a non-singular projective surface $X$ onto a non-singular projective curve $B$ with connected fibers.
A general fiber of $f$ is a smooth curve of genus $g\geq 2$.
The fibration is said to be relatively minimal if there is no $(-1)$-curve contained in the fibers of $f$.
Here a curve $C$ is called a $(-k)$-curve
if it is a smooth rational curve with self-intersection $C^2=-k$.
The fibration is called hyperelliptic if its general fiber is a hyperelliptic curve,
smooth if all its fibers are smooth,
isotrivial if all its smooth fibers are isomorphic to each other,
locally trivial if it is both smooth and isotrivial,
and semi-stable if all its singular fibers are reduced nodal curves.

The relative canonical sheaf of $f$ is defined to be $\omega_{f}=\omega_X\otimes f^*\omega_B^{\vee}$,
where $\omega_X$ (resp. $\omega_B$) is the canonical sheaf of $X$ (resp. $B$).
For a relatively minimal fibration $f$, the relative canonical sheaf $\omega_f$ is numerical effective (nef),
i.e., $\omega_f\cdot C \geq 0$ for any curve $C\subseteq X$.
Let $b=g(B)$, $p_g=h^0(X,\,\omega_X)$, $q=h^1(X,\,\omega_X)$,
$\chi(\mathcal O_X)=p_g-q+1$, and $\chit(X)$ be the topological Euler characteristic of $X$.
The basic invariants of $f$ are:
\begin{equation}\label{eqn-invarians-f}
\left\{\begin{aligned}
\chi_f&=\deg f_*\omega_{f}=\chi(\mathcal O_X)-(g-1)(b-1),\\
\omega_{f}^2&=\omega_X^2-8(g-1)(b-1),\\
e_f&=\chit(X)-4(g-1)(b-1).
\end{aligned}\right.
\end{equation}
We will always assume that $f$ is relatively minimal. Under this assumption, these invariants satisfy the following properties:
\begin{eqnarray}
&&12\chi_f=\omega_f^2+e_f.\label{eqnnoether}\\
&&e_f\geq 0;~\text{moreover, $e_f=0$ iff $f$ is smooth}.\nonumber\\
&&\chi_f\geq 0;~\text{moreover, $\chi_f=0$ iff $f$ is locally trivial}.\qquad\nonumber
\end{eqnarray}

If $f$ is not locally trivial, the slope of $f$ is defined to be
$$\lambda_f=\frac{\omega_f^2}{\chi_f}.$$
It follows immediately that $0< \lambda_f\leq 12$.
The main known result is the slope inequality:
\begin{theorem}[Cornalba-Harris-Xiao, \cite{cornalba-harris-88,xiao-87a}]\label{thm-chx}
If $f$ is not locally trivial, then $$\lambda_f\geq \frac{4(g-1)}{g}.$$
\end{theorem}

Moreover,
the equality in the above lower bound can hold only for the hyperelliptic fibrations (cf. \cite{cornalba-harris-88,konno-93,stoppino-08}).
Thus, it is natural to investigate the influence of some properties of the fibration on the behaviour of the slope.
For instance, according to \cite{konno-99,barja-stoppino-08},
one knows that the Clifford index of the general fiber
has some meaning to the lower bound of the slope.
We would like to be concerned about
the following conjecture of Barja and Stoppino
(cf. \cite[Conjecture\,1.1]{barja-stoppino-08})
on the influence of the relative irregularity $q_f:=q-b$ on the lower bound of the slope.
\begin{conjecture}[Barja-Stoppino]\label{conjecturebs}
If $f$ is not locally trivial and $q_f < g-1$, then
\begin{equation}\label{conjectureequ}
\lambda_f\geq \frac{4(g-1)}{g-q_f}.
\end{equation}
\end{conjecture}
The first result in the direction is due to Xiao \cite[Theorem\,3]{xiao-87a}, where he proved
that if $q_f > 0$, then $\lambda_f \geq 4$ and the equality can hold only when $q_f = 1$.
In \cite[Theorem\,1.3]{barja-stoppino-08},
Barja and Stoppino considered the influence of the Clifford index ${\rm Cliff}(f)$
of the general fiber and
the relative irregularity $q_f$ on the lower bound of the slope simultaneously,
and proved that
\begin{equation*}
\lambda_f\geq \frac{4(g-1)}{g-[m/2]},
\end{equation*}
where $m=\min\big\{{\rm Cliff}(f),\,q_f\big\}$
and $[\bullet]$ stands for the integral part.
When the Clifford index ${\rm Cliff}(f)$ is large,
this shows that the lower bound $\lambda_f$ is increasing
with the relative irregularity $q_f$
and it is close to the conjectured bound.
In \cite[Corollary\,1.5]{lu-zuo-13}, we proved the above conjecture for hyperelliptic fibrations.
This conjecture remains open in the general case.

Our first main result is a lower bound on the slope, which increases with the relative irregularity $q_f$.
\begin{theorem}\label{thm-main-1}
	Let $f$ be a fibration of genus $g\geq 2$ which is not locally trivial. If $q_f>0$, then
	\begin{equation}\label{eqn-main-1}
	\lambda_f> \frac{4(g-1)}{g-q_f/2}.
	\end{equation}
\end{theorem}
Note that the above lower bound improves Barja-Stoppino's \cite{barja-stoppino-08}.
Our next main result is towards \autoref{conjecturebs}.
%
\begin{theorem}\label{thm-main-2}
Let $f$ be a fibration of genus $g\geq 2$ which is not locally trivial.
\begin{enumerate}
\item[(i)] If $q_f \leq g/9$, then \eqref{conjectureequ} holds.
\item[(ii)] If $g$ is odd and $q_f=(g+1)/2$, then there exist fibrations violating \eqref{conjectureequ}.
\end{enumerate}
\end{theorem}

Pirola constructed in \cite{pirola-92} the first example which does not satisfy \eqref{conjectureequ},
see also \cite[Remark\,4.6]{barja-stoppino-08}.
To our knowledge, the only known counterexamples to the bound \eqref{conjectureequ} belong to the extremal case $q_f=g-1$.
According to \cite[Corollary\,4]{xiao-87a}, the genus of fibrations with $q_f=g-1$ is bounded from above ($g\leq 7$).
In our construction of the counter examples, the genus has no upper bound.

\vspace{2mm}
The main idea of the proof of the lower bound on the slope is a combination of
Xiao's technique \cite{xiao-87a} and the second multiplication map.
Such a combination has been already applied to study the influence of the gonality of a general fiber
on the lower bound of the slope and the Severi problem \cite{lu-zuo-15-0,lu-zuo-15b}. 
It turns out that the theorem follows from the combination of these two techniques
if the fibration $f$ is not a double cover fibration. 
Hence we are reduced to study the double cover fibrations.

Double cover fibrations have already been studied earlier by many authors,
see \cite{barja-01,barja-zucconi-01,cornalba-stoppino-08,stoppino-08} etc.
We first define certain local relative invariants for a double cover fibration and show that
the basic invariants as in \eqref{eqn-invarians-f} can be expressed by these local relative invariants
and relative invariants of the quotient fibration (cf. \autoref{thminvariants-double-fibration}).
Then we study influence of the irregularity of the double cover on these local relative invariants
with the help of the Albanese map (cf. \autoref{prop-restriction-invariants}),
which enables us to deduce the required lower bounds on the slope of a double cover fibration.

\vspace{2mm}
Our paper is organized as follows.
In \autoref{sec-main}, we prove the lower bounds
on the slope \big(\autoref{thm-main-1} and \autoref{thm-main-2}\,(i)\big).
In \autoref{sec-non-hyper}, we mainly study the lower bound on the slope of the non-double cover fibrations using a
a combination of Xiao's technique \cite{xiao-87a} and the second multiplication map.
In \autoref{sec-double} we consider the lower bound on the slope of the double cover fibrations.
Finally, in \autoref{sec-examples} we provide the counterexamples to \eqref{conjectureequ}.

\section{Proof of the lower bounds}\label{sec-main}
In this section, we prove the lower bounds on the slope,
i.e., we prove \autoref{thm-main-1} and \autoref{thm-main-2}\,(i).
It is based on certain technical lemmas, which will be proved later.

\begin{definition}\label{def-double-cover}
The fibration $f$ is said to be a double cover fibration of type $(g, \gamma)$
if there is a fibration $h':\,Y' \to B$ and a rational map
$\pi:\,X \dashrightarrow Y'$ ($Y'$ may be singular)
such that the general fiber of $h'$ is a genus-$\gamma$ curve, $\deg\pi=2$ and $h'\circ\pi=f$.
$$\xymatrix{
X \ar@{-->}[rr]^-{\pi}\ar[dr]_-{f} && Y'\ar[dl]^-{h'}\\
&B&
}$$
We remark that there might exist more than one double cover fibration structure on a given double cover fibration.
\end{definition}

\begin{definition}\label{def-hn-filtra}
For any locally free sheaf $\mathcal E$ on a smooth projective curve $B$,
the slope of $\mathcal E$ is defined to be the rational number
$\mu(\mathcal E)=\deg(\mathcal E)/{\rank(\mathcal E)}.$
The sheaf $\mathcal E$ is said to be semi-stable,
if for any coherent subsheaf $0\neq\mathcal E'\subsetneq\mathcal E$
we have $\mu(\mathcal E')\leq\mu(\mathcal E)$.
The Harder-Narasimhan (H-N) filtration of $\mathcal E$ is the following unique filtration:
 \begin{equation}\label{eqnharder-nara}
 0=\mathcal E_0\subset \mathcal E_1 \subset \cdots \subset \mathcal E_n=\mathcal E,
 \end{equation}
 such that:
\begin{list}{}
{\setlength{\labelwidth}{0.6cm}
\setlength{\leftmargin}{0.7cm}}
 \item[(i)] the quotient $\mathcal E_i/\mathcal E_{i-1}$ is a locally free semi-stable sheaf for each $i$;
 \item[(ii)] the slopes are strictly decreasing
 $\mu(\mathcal E_i/\mathcal E_{i-1})>\mu(\mathcal E_j/\mathcal E_{j-1})$ if $i>j$.
\end{list}
The H-N filtration always exists.
In particular, the H-N filtration exists for $\mathcal E=f_*\omega_f$,
and in this case we write
$$\mu_i=\mu(\mathcal E_i/\mathcal E_{i-1}),\qquad
r_i=\rank(\mathcal E_i), \qquad \delta=g-r_{n-1}.$$
By definition one has
$$\delta \geq q_f.$$
\end{definition}

\begin{lemma}\label{thm-3-1}
Let $f$ be a locally non-trivial non-hyperelliptic fibration of genus $g\geq 3$.
Assume that either $f$ is not a double cover fibration, or $f$ is a double cover fibration such that
$\gamma\geq g/4$ for any possible double cover fibration structure of type $(g,\gamma)$ on $f$.
If $\mu_n=0$, then
\begin{equation}\label{eqn-3-thm}
\lambda_f >
\left\{\begin{aligned}
&\frac{18g-47\delta}{4g-11\delta}\cdot \frac{g-1}{g+1},&\qquad&\text{if~}~\delta\leq \frac{2g}{21};\\[1mm]
&\frac{72g-46\delta}{16g-13\delta}\cdot \frac{g-1}{g+1},&&\text{if~}~\frac{2g}{21}\leq \delta \leq \frac{4g}{7}.
\end{aligned}\right.
\end{equation}
\end{lemma}

\begin{lemma}\label{thm-3-2}
Let $f$ be the same as in \autoref{thm-3-1}.
If $\delta\geq \frac{2(g+8)}{9}$, then
\begin{equation}\label{eqn-thm-3-2}
\lambda_f> \frac{4(g-1)}{g-\delta/2}.
\end{equation}
\end{lemma}

\begin{lemma}\label{prop-2-21}
	Let $f:\,X \to B$ be a locally non-trivial double cover fibration of type $(g,\gamma)$ with $g\geq 4\gamma+1$,
	and $h:\,Y\to B$ be the associated quotient fibration as in \autoref{figure-5-1}.
	Assume that either $\gamma=1$, or $h$ is locally trivial, or
	$$\lambda_h> \frac{4(\gamma-1)}{\gamma-q_h/2}.$$
	Then
	\begin{equation}\label{eqn-2-22}
	\lambda_f> \frac{4(g-1)}{g-q_f/2}.
	\end{equation}
\end{lemma}

\begin{lemma}\label{thm-3-4}
	Let $f$ be a locally non-trivial non-hyperelliptic fibration of genus $g\geq 3$.
    If $q_f\leq g/2$ and $f$ is a double cover fibration of type $(g,\gamma)$ with $g\geq 4\gamma-2$,
    then $\lambda_f\geq \frac{4(g-1)}{g-q_f}$.
\end{lemma}

The proofs of the above four technical lemmas will be postponed in Sections \ref{sec-pf-thm-3-1},
\ref{sec-pf-thm-3-2}, \ref{sec-pf-prop-2-21}, \ref{sec-pf-thm-3-4} respectively.
Based on the above lemmas, we will prove the lower bounds on the slope of fibrations with positive irregularity.

\begin{proposition}\label{thm-3-3}
	Let $f$ be a locally non-trivial non-hyperelliptic fibration of genus $g\geq 3$.
	Assume that either $f$ is not a double cover fibration, or $f$ is a double cover fibration such that
	$\gamma-1\geq (g-1)/4$ for any possible double cover fibration structure of type $(g,\gamma)$ on $f$.
	If $q_f\neq 0$, then
	\begin{equation}\label{eqn-2-21}
	\lambda_f\geq \frac92.
	\end{equation}
\end{proposition}
\begin{proof}
     Because $q_f\neq 0$, we may construct \'etale covers of $X$ which are still fibred over $B$:
     $$\xymatrix{\wt X \ar[rr]^-{\pi} \ar[dr]_{\tilde f} && X \ar[dl]^{f}\\
     	&B&}$$
     Since $\pi$ is \'etale, the induced fibration $\tilde f$ is still not trivial and $\lambda_{\tilde f}=\lambda_f$.
     Moreover, by Riemann-Hurwitz formula one has
     $$\tilde g=\deg\pi\cdot (g-1)+1, \qquad \text{where $\tilde g$ is the genus of a general fiber of $\tilde f$}.$$
     In fact, we can even construct a Galois \'etale cover $\pi$ with $\deg \pi$ being prime.
     
     We claim that
     \begin{quote}
     	If $\pi$ is a Galois \'etale cover such that $\deg \pi$ is prime and sufficiently large,
     	then either $\tilde f$ is not a double cover fibration, or $\tilde f$ is a double cover fibration such that
     	$\tilde \gamma-1\geq (\tilde g-1)/4$ for any possible double cover fibration structure of
     	type $(\tilde g,\tilde \gamma)$ on $\tilde f$.
     \end{quote}
     Assume the above claim. Then \eqref{eqn-2-21} follows immediately by applying \autoref{thm-3-1} to the new fibration $\tilde f$.
     It remains to prove the above claim.
     
     We prove the above claim by contradiction
     If $\tilde f$ is a double cover fibration of type $(\tilde g,\tilde \gamma)$ with $\tilde \gamma-1< (\tilde g-1)/4$,
     then there is an involution $\tilde \sigma$ on $\wt X$.
     Let $G$ be the automorphism subgroup of $\wt X$ induced by the Galois cover $\pi$,
     and $\wt G$ the automorphism subgroup generated by $G$ and $\tilde \sigma$.
     If $G$ is normal in $\wt G$, then $\tilde \sigma$ induces an involution on $X$, which realizes $X$ as a double cover fibration
     of type $(g,\gamma)$ with $\gamma-1< (g-1)/4$, contradicting the assumption.
     Hence $G$ is not normal in $\wt G$. Since $p:=|G|=\deg \pi$ is prime, it follows that
     $\wt G\geq p(p+1)$ by Sylow's theorem.
     However, when $p$ is large,
     this contradicts the linear bound on the automorphism group of curves (cf. \cite[Exercise\,IV.2.5]{hartshorne-77}):
     indeed, it is clear that $\wt G$ acts faithfully on the general fiber of $\tilde f$, from which it follows that
     $$p(p+1)\leq |\wt G| \leq 84(\tilde g-1)=84p(g-1).$$
     This gives a contradiction when $p\geq 84(g-1)$.
     Thus we complete the proof of the claim, and hence also the proposition.
     \end{proof}

\begin{proof}[{Proof of \autoref{thm-main-1}}]
	We prove by induction on the genus $g$.
	
	When $g=2$, then $f$ is hyperelliptic.
	Hence \eqref{eqn-main-1} follows from \cite[Corollary\,1.5]{lu-zuo-13}.
	
	We now assume that $g>2$.	
	If either $f$ is not a double cover fibration, or $f$ is a double cover fibration such that
	$\gamma\geq g/4$ for any possible double cover fibration structure of type $(g,\gamma)$ on $f$,
	then \eqref{eqn-main-1} follows directly from \eqref{eqn-thm-3-2} since $\delta \geq q_f$ by definition.
	Thus we may assume that $f$ is a double cover fibration of type $(g,\gamma)$ with $g\geq 4\gamma+1$.
	Let $h:\,Y\to B$ be the associated quotient fibration as in \autoref{figure-5-1}.
	By induction, we may assume that
		$$\lambda_h> \frac{4(\gamma-1)}{\gamma-q_h/2},\qquad\text{if $\gamma\geq 2$ and $h$ is locally non-trivial}.$$
	Hence according to \autoref{prop-2-21}, one proves \eqref{eqn-main-1}.
\end{proof}
\begin{proof}[Proof of \autoref{thm-main-2}\,{\rm(i)}]
	First by \autoref{thm-chx}, we may assume that $q_f>0$.
	
	Consider next the case when $f$ is not a double cover fibration, or when $f$ is a double cover fibration such that
	$\gamma-1\geq (g-1)/4$ for any possible double cover fibration structure of type $(g,\gamma)$ on $f$.
	Then \eqref{conjectureequ} follows from \eqref{eqn-2-21} since $q_f\leq g/9$.
	
	Finally, we consider the case when $f$ is a double cover fibration of type $(g,\gamma)$ with $g\geq 4\gamma-2$.
	In this case, \eqref{conjectureequ} follows from \autoref{thm-3-4}.
\end{proof}

\begin{remarks}
	(i) The assumption $q_f\leq g/9$ in \autoref{thm-main-2}\,{\rm(i)} might be relaxed a little.
	But the proof requires a much more complicated computation.
	
	(ii) We only deal with the case when $q_f$ is small with respect to $g$. If $q_f$ is big,  we refer to
	\cite[Theorem\,3.2]{barja-zucconi-01} for a similar lower bound on the slope.
\end{remarks}

\section{Slope of non-hyperelliptic fibrations}\label{sec-non-hyper}
In this section, we consider the lower bound on the slope of the non-hyperelliptic fibrations
and double cover fibrations of type $(g,\gamma)$ with $g$ is not big with respect to $\gamma$ (e.g., $g\leq 4\gamma$).
The main techniques are Xiao's technique \cite{xiao-87a} and the second multiplication map.
We first review these two techniques in \autoref{sec-pre};
and then prove \autoref{thm-3-1} (resp. \autoref{thm-3-2}) in \autoref{sec-pf-thm-3-1} (resp. \autoref{sec-pf-thm-3-2}).

\subsection{Preliminaries}\label{sec-pre}
In this subsection, we briefly review
Xiao's technique \cite{xiao-87a} and the second multiplication map developed in \cite{lu-zuo-15-0}.
Both techniques are based on the Harder-Narasimhan (H-N) filtration on the direct image sheaf $f_*\omega_f$,
which we recall first.

Let $\mathcal E$ be a (non-zero) locally free sheaf over $B$.
It is said to be positive (resp. semi-positive), if for any quotient sheaf  $\mathcal E \twoheadrightarrow \mathcal Q \neq 0$, one has $\deg \mathcal Q >0$ (resp. $\deg \mathcal Q \geq0$).
Define
$$\mu_f(\mathcal E)=\max\{\deg \mathcal F~|~\mathcal E \otimes \mathcal F^{\vee} \text{~is semi-positive}\}.$$
Then $\mathcal E$ is positive (resp. semi-positive) if and only if $\mu_f(\mathcal E)>0$ (resp. $\mu_f(\mathcal E)\geq0$).

It is easy to see that $\mu_f(\mathcal E_i)=\mu_i$.
In particular, $\mu_f(f_*\omega_f)=\mu_n\geq 0$ due to the semi-positivity of $f_*\omega_f$.
Moreover, one has
\begin{equation}\label{eqn-degree-chi_f}
\chi_f=\sum_{i=1}^{n}r_i(\mu_i-\mu_{i+1}), \quad\text{where~}r_i:=\rank \mathcal E_i\text{~and~}\mu_{n+1}:=0.
\end{equation}

\begin{definition}[\cite{xiao-87a}]\label{defofN(F)}
Let $\mathcal E'$ be any locally free subsheaf of $f_*\omega_f$.
The fixed and moving parts of $\mathcal E'$, denoted by $Z(\mathcal E')$ and $M(\mathcal E')$ respectively, are defined as follows.
Let $\call$ be a sufficiently ample line bundle on $B$ such that the sheaf $\mathcal E'\otimes\call$ is generated by its global sections,
and $\Lambda(\mathcal E')\subseteq |\omega_f\otimes f^*\call|$ be the linear subsystem corresponding to sections in $H^0(B,\,\mathcal E'\otimes\call)$.
Then we define $Z(\mathcal E')$ to be the fixed part of $\Lambda(\mathcal E')$, and $M(\mathcal E')=\omega_f-Z(\mathcal E')$.
Note that the definitions do not depend on the choice of $\call$.
\end{definition}

For a general fiber $F$ of $f$, let
\begin{equation}\label{eqn-def-iota_i}
\iota_i:~F \lra \Gamma_i \subseteq \mathbb P^{r_i-1}
\end{equation}
be the map defined by the restricted linear subsystem $\Lambda(\mathcal E_i)\big|_{F}$ on $F$ if $r_i\neq 1$,
where $\mathcal E_i\subseteq f_*\omega_f$ is any subsheaf in
the H-N filtration of $f_*\omega_f$ in \eqref{eqnharder-nara}.
Let $d_i=M(\mathcal E_i)\cdot F$, and $\gamma_i$ be the geometric genus of $\Gamma_i$.
For convention, we define $d_{n+1}=2g-2$.
It is clear that $\iota_i$ factors through $\iota_j$ if $i\leq j$, from which it follows that
\begin{equation}\label{eqn-factor-ij}
\left\{\begin{aligned}
&\text{$\deg(\iota_j)$ divides $\deg(\iota_i)$, $d_j\geq d_i$ and $\gamma_j\geq \gamma_i$,}\quad \forall~i\leq j;\\
&\text{moreover, $\gamma_i=\gamma_j$ if $\deg(\iota_i)=\deg(\iota_j)$.}
\end{aligned}\right.
\end{equation}

\begin{lemma}\label{lemma-d_i}
If $\iota_i$ is not birational, then
\begin{equation}\label{eqn-d_i-non-bi}
d_i\geq \deg(\iota_i)\cdot \min\big\{2(r_i-1),~r_i+\gamma_i-1\big\}.
\end{equation}
If $\iota_i$ is birational, then
\begin{equation}\label{eqn-d_i-bi}
d_i\geq \min\left\{3r_i-5,~\frac{g}{2}+\frac{3r_i}{2}-2\right\}.
\end{equation}
\end{lemma}
\begin{proof}
Let $\tau_i:\,\wt \Gamma_i \to \Gamma_i$ be the normalization, and $D_i=\tau_i^*\big(\mathcal O_{\mathbb P^{r_i-1}}(1)\big)\in\Pic\big(\wt \Gamma_i\big)$
be the pulling-back of the hyperplane section. Then \eqref{eqn-d_i-non-bi}
follows from the facts that $d_i=\deg(\iota_i)\cdot \deg(D_i)$, and
$$\deg(D_i)\geq \left\{\begin{aligned}
&h^0\big(\wt\Gamma_i,\,D_i\big)+\gamma_i-1\geq r_i+\gamma_i-1,&\quad&\text{if~}h^1\big(\wt\Gamma_i,\,D_i\big)=0;\\
&2\left(h^0\big(\wt\Gamma_i,\,D_i\big)-1\right)\geq 2(r_i-1), &&\text{if~}h^1\big(\wt\Gamma_i,\,D_i\big)\neq0.
\end{aligned}\right.$$
Note that we use Clifford's theorem on special divisors above.

To prove \eqref{eqn-d_i-bi}, we apply Castelnuovo's bound (cf. \cite[\S\,III.2]{acgh-85}) which asserts that
\begin{equation}\label{eqn-castelnuovo}
d_i\geq \frac{g}{m_i}+\frac{(m_i+1)}{2}\cdot s_i -m_i\geq \frac{g}{m_i}+\frac{(m_i+1)}{2}\cdot r_i -m_i,
\end{equation}
where $s_i=h^0\big(F,\,M(\mathcal E_i)|_F\big)\geq r_i$ and $m_i=\left[\frac{d_i-1}{s_i-2}\right]$.
Hence \eqref{eqn-d_i-bi} follows immediately.
\end{proof}

\begin{lemma}\label{lem-2-1}
Assume that either $\deg(\iota_i)\neq 2$, or $\deg(\iota_i)=2$ and $\gamma_i\geq g/6$.
If $d_i< g-1$, then $d_i\geq 3(r_i-1)$.
\end{lemma}
\begin{proof}
It is clear if $\deg(\iota_i)\geq 3$.
If $\deg(\iota_i)=2$, then by \eqref{eqn-d_i-non-bi} together with the assumption $\gamma_i\geq g/6$,
one obtains $$g-2\geq d_i\geq \min\big\{4(r_i-1), 2(r_i-1)+g/3\big\},
\quad\Longrightarrow\quad g\geq 3r_i.$$
Hence $d_i\geq \min\big\{4(r_i-1), 2(r_i-1)+g/3\big\}\geq 3(r_i-1)$.

If $\deg(\iota_i)=1$, then $r_i\geq 3$, and according to Castelnuovo's bound \eqref{eqn-castelnuovo} one has
$$d_i\geq \left\{
\begin{aligned}
&m_i(r_i-2)+1 \geq 3r_i-3, &\quad&\text{if~}m_i\geq 5;\\
&4r_i-7 \geq 3r_i-3, &\quad&\text{if~}m_i=4\text{~and~}r_i\geq 4;\\
&\frac{g}{3}+2r_i-3, \quad\Longrightarrow\quad d_i> 3r_i-4,&\quad&\text{if~}m_i=3;\\
&\frac{g}{2}+\frac{3r_i}{2}-2, \quad\Longrightarrow\quad d_i> 3r_i-3,&\quad&\text{if~}m_i=2.
\end{aligned}\right.$$
We use the assumption $g> d_i+1$ when $m_i=3$ or $2$ above.
To complete the proof, it remains to consider the case when $m_i=4$ and $r=r_i=3$.
As $\iota_i$ is birational, by the genus formula for plane curves, one obtains that
$$d_i+1< g\leq \frac{(d_i-1)(d_i-2)}{2},$$
from which it follows that $d_i\geq 6=3(r_i-1)$ as required.
\end{proof}

\begin{remark}\label{rem-2-2}
Assume that either $\deg(\iota_i)\neq 2$, or $\deg(\iota_i)=2$ and $\gamma_i\geq g/6$.
If $d_i=g-1$ or $g$, then one can show similarly that $d_i\geq 3r_i-4$.
\end{remark}

\begin{corollary}\label{lemma-d_i-2}
Assume that either $\deg(\iota_i)\neq 2$, or $\deg(\iota_i)=2$ and $\gamma_i\geq g/6$.
If $r$ is an integer such that $r_i\geq r$ and $g>3(r-1)$,
then $d_i\geq 3(r-1)$.
\end{corollary}
\begin{proof}
Assume that $d_i<3(r-1)\leq 3(r_i-1)$.
Hence by \autoref{lem-2-1}, $d_i\geq g-1$.
Thus $3(r-1)\geq g$, which contradicts the assumption.
\end{proof}

The next proposition, which is due to Xiao, is crucial to the study of the slope of fibrations.
\begin{proposition}[\cite{xiao-87a}]\label{prop-xiao}
For any sequence of indices $1\leq i_1 <\cdots <i_{k}\leq n$, one has
\begin{equation}\label{eqn-xiao}
\omega_f^2\geq \sum_{j=1}^{k}\big(d_{i_j}+d_{i_{j+1}}\big)\big(\mu_{i_j}-\mu_{i_{j+1}}\big),\quad\text{where $i_{k+1}=n+1$.}
\end{equation}
In particular, one has
\begin{equation}\label{eqn-xiao-1}
\omega_f^2\geq \sum_{i=1}^{n}\big(d_{i}+d_{i+1}\big)\big(\mu_{i}-\mu_{i+1}\big).
\end{equation}
\end{proposition}

\begin{corollary}\label{cor-xiao-l-delta}
If $\mu_n=0$, then
\begin{equation}\label{eqn-cor-xiao}
\omega_f^2> \frac{(2g-2)^2}{(2g-2)\cdot r_{n-1}-d_i\cdot (r_{n-1}-r_{i-1})}\cdot \chi_f,\qquad \forall~1< i < n.
\end{equation}
\end{corollary}
\begin{proof}
According to \eqref{eqn-degree-chi_f}, one has
$$\chi_f \leq \sum_{j=1}^{i-1}r_i(\mu_j-\mu_{j+1})+\sum_{j=i}^{n-1}r_{n-1}(\mu_j-\mu_{j+1}) =r_{i-1}\cdot \mu_1+(r_{n-1}-r_{i-1})\cdot \mu_i.$$
By \eqref{eqn-xiao}, one has
$$\begin{aligned}
\omega_f^2&~\geq (d_1+d_i)\cdot (\mu_1-\mu_i)+ (2g-2+d_i)\cdot \mu_i \geq d_i\cdot \mu_1+(2g-2)\cdot \mu_i.
\end{aligned}$$
Combining the above inequalities together with Konno's bound \cite[(2.6)]{konno-94}
\begin{equation}\label{eqn-konno}
\omega_f^2>(2g-2)\mu_1,
\end{equation}
one gets
$$\left(\frac{r_{n-1}-r_{i-1}}{2g-2}+\frac{r_{i-1}-d_i\cdot \frac{r_{n-1}-r_{i-1}}{2g-2}}{2g-2}\right)\cdot\omega_f^2> \chi_f.$$
By rearrangement, we obtain \eqref{eqn-cor-xiao}.
\end{proof}

The next proposition on the lower bound of $\omega_f^2$ is based on the second multiplication map (cf. \cite[\S\,2.2]{lu-zuo-15-0}):
$$\varrho:\,S^2(f_*\omega_f) \lra f_*\big(\omega_f^{\otimes 2}\big).$$
\begin{proposition}\label{prop-second-mult}
Assume that the general fiber $F$ is non-hyperelliptic, $\iota_{n-1}$ is birational and $\mu_n=0$. Then
\begin{equation}\label{eqn-second-mult}
\omega_f^2\geq \sum_{i=1}^{n-1}(2\theta_i-r_i)(\mu_i-\mu_{i+1})+\sum_{i=\tilde l}^{n-1}\tilde \theta_i(\mu_i-\mu_{i+1}),
\end{equation}
where
\begin{eqnarray}
\tilde l&=&\min\left\{i~\Big|~r_i+g\geq 2r_{n-1},~\iota_i
\text{~is birational, and~}r_i\geq \frac{g}{3}+2\right\};\label{eqn-def-tilde-l}\\[1mm]
\theta_i&=&\left\{\begin{aligned}
&1&\quad&\text{if~}i=1\text{~and~}r_1=1,\\
&\min\{3r_i-3,\,2r_i+\gamma_i-1\},&& \text{otherwise};
\end{aligned}\right.\label{eqn-def-theta_i}\\[1mm]
\tilde\theta_i&=&\frac{3}{2}(r_i+g-2r_{n-1}).\label{eqn-def-t-theta_i}
\end{eqnarray}
\end{proposition}
\begin{proof}
Let
$$\mu_i'=\max\{2\mu_i,\,\mu_{\tilde l}\},\qquad \forall~1\leq i\leq n.$$
By assumption, one has
$$\mu_{n}'=\mu_{\tilde l},\quad \theta_{n-1}=3r_{n-1}-3,\quad \tilde\theta_i=\frac{3}{2}(r_i+g-2)-\theta_{n-1}.$$
According to \cite[Proposition\,2.4 \& Lemma\,2.5]{lu-zuo-15-0} and \autoref{lemma-second-2} below with
the decreasing sequence
$$\big\{2\mu_1,\cdots,\,\cdots,\,2\mu_{n-1},\,\mu_{\tilde l},\,\cdots,\,\mu_{n-1}\big\},$$
and the increasing sequence
$$\big\{\theta_1,\,\cdots,\,\theta_{n-1},\, \theta_{n-1}+\tilde\theta_{\tilde l},\,\cdots,\,\theta_{n-1}+\tilde\theta_{n-1}\big\},$$
we obtain (we set $\theta_0=0$)
\begin{eqnarray*}
\omega_f^2+\chi_f &\geq& \sum_{i=1}^{n-1}\theta_i\big(\mu_{i}'-\mu_{i+1}'\big)+\sum_{i=\tilde l}^{n-1}\big(\theta_{n-1}+\tilde\theta_i\big)\big(\mu_{i}-\mu_{i+1}\big)\\
&=&\sum_{i=1}^{n-1}(\theta_{i}-\theta_{i-1})\mu_{i}'-\theta_{n-1}\cdot \mu_n'+\sum_{i=\tilde l}^{n-1}\big(\theta_{n-1}+\tilde\theta_i\big)\big(\mu_{i}-\mu_{i+1}\big)\\
&\geq&\sum_{i=1}^{n-1}(\theta_{i}-\theta_{i-1})\cdot 2\mu_{i}-\theta_{n-1}\cdot \mu_{\tilde l}+\sum_{i=\tilde l}^{n-1}\big(\theta_{n-1}+\tilde\theta_i\big)\big(\mu_{i}-\mu_{i+1}\big)\\
&=&\sum_{i=1}^{n-1}2\theta_i\big(\mu_{i}-\mu_{i+1}\big)+\sum_{i=\tilde l}^{n-1}\tilde\theta_i\big(\mu_{i}-\mu_{i+1}\big).
\end{eqnarray*}
Hence \eqref{eqn-second-mult} follows from the above inequality together with \eqref{eqn-degree-chi_f}.
\end{proof}

\begin{lemma}\label{lemma-second-2}
If $\iota_i$ is birational,
then there exists a subsheaf $\mathcal F_{i}\subseteq f_*\big(\omega_f^{\otimes 2}\big)$
such that
\begin{equation}\label{eqn-2-2}
\mu_f(\mathcal F_{i})\geq \mu_i+\mu_n,
\qquad \rank\mathcal F_{i}\, \geq  g+d_i+r_i-1-h^0\big(F,\,M(\mathcal E_i)|_F\big),
\end{equation}
where $M(\mathcal E_i)$ is defined in \autoref{defofN(F)}.
In particular,  if $\iota_i$ is birational and $r_i\geq \frac{g}{3}+2$, then
there exists a subsheaf $\mathcal F_{i}\subseteq f_*\big(\omega_f^{\otimes 2}\big)$
such that
\begin{equation}\label{eqn-2-3}
\mu_f(\mathcal F_{i})\geq \mu_i+\mu_n,
\qquad \rank\mathcal F_{i}\, \geq  \frac{3}{2}(r_i+g-2).
\end{equation}
\end{lemma}
\begin{proof}
Let $\mathcal E_i\subseteq \mathcal E= f_*\omega_f$ be any subsheaf in the H-N filtration of $f_*\omega_f$ in \eqref{eqnharder-nara}.
Consider the composition map
$$\varrho_{i}:\,\mathcal E_i \otimes \mathcal E \lra S^2\big(f_*\omega_f\big) \lra f_*\big(\omega_f^{\otimes 2}\big).$$
It is clear that
$\mu_f \big(\im\,(\varrho_{i})\big)\geq \mu_f\big(\mathcal E_i\big)+\mu_f(\mathcal E)\geq \mu_i$.
To prove \eqref{eqn-2-2}, it suffices to show that
\begin{equation}\label{eqn-pf-pre-large-1}
\rank \big(\im\,(\varrho_{i})\big)\geq g+d_i+r_i-1-h^0\big(F,\,M(\mathcal E_i)|_F\big).
\end{equation}
Similar to \cite[Lemma\,2.5]{lu-zuo-15-0},
\eqref{eqn-pf-pre-large-1} follows from the next lemma since $\iota_i$ is birational.
Hence \eqref{eqn-2-2} is proved.
And \eqref{eqn-2-3} follows from \eqref{eqn-2-2} together with Castelnuovo's bound \eqref{eqn-castelnuovo}.
The proof is complete.
\end{proof}

\begin{lemma}
Let $D\in \Pic(Z)$ be an effective divisor of a smooth curve $Z$ of genus $g$,
$V\subseteq H^0(Z,D)$ be a subspace with $\dim V=r$, and
$$\rho:~V \otimes H^0(Z,\,K_Z) \lra H^0(Z,\,K_Z+D)$$
be the natural multiplication map, where is $K_Z$ is the canonical divisor of $Z$.
Assume that $D\subseteq K_Z$ and $V$ induces a birational map $\phi_V$ on $Z$.
Then
\begin{equation}\label{eqn-pre-1}
\dim \big(\im(\rho)\big)\geq g+\deg D+r-1-h^0(Z,\,D).
\end{equation}
\end{lemma}
\begin{proof}
Since $\phi_V$ is birational, the complete linear system $|D|$ automatically defines a birational map $\phi_{D}$, and one has the following commutative diagram \big($s=h^0(Z,\,D)$\big).
$$\xymatrix{Z \ar[rr]^{\phi_{D}}\ar[rd]_-{\phi_{V}} &&\mathbb P^{s-1} \ar@{-->}[ld]\\
&\mathbb P^{r-1}&
}$$
According to the general position theorem (cf. \cite[\S\,III.1]{acgh-85}),
there exist $s$ points $\{p_1,\cdots,p_{s}\}\subseteq Z$ such that any $s-1$ of them
give linearly independent conditions for the vector space $H^0(Z,\,D) \,\big(\supseteq V\big)$.
Hence there exist
$\{v_1,\cdots v_{r}\}\subseteq V$ such that
$$v_j(p_j)\neq 0,\quad\text{but}\quad v_j(p_i)=0,~\forall\, 1\leq i \leq r \text{~and~} i\neq j.$$
Let $V_{12}\subseteq V$ be generated by $v_1$ and $v_2$.
Consider the subspace
\begin{equation}\label{eqn-2-1}
W\triangleq\langle v_3^2,\cdots, v_{r}^2 \rangle \subseteq H^0(Z,\,2D) \hookrightarrow H^0(Z,\,K_Z+D),
\end{equation}
and the restriction map
$$\varphi:~V_{12}\,\otimes H^0(Z,\,K_Z) \lra H^0(Z,\,K_Z+D).$$
According to the base-point-free pencil trick (cf. \cite[\S\,III.3]{acgh-85}), one checks easily that
$$\begin{aligned}
\dim \im(\varphi) &= 2g-h^0\big(Z,\,K_Z-(D-p_3-\cdots-p_{r})\big)\\
&=2g-\Big(h^0\big(Z,\,(D-p_3-\cdots-p_{r})\big)+r+g-3-\deg D\Big)\\
&=g+1+\deg D-h^0(Z,\,D).
\end{aligned}$$
The last step follows from the fact that
$$h^0\big(Z,\,(D-p_3-\cdots-p_{r})\big)=h^0(Z,\,D)-(r-2),$$
since $\{p_3,\cdots,p_r\}$ are in general position.
Note that $\dim W=r-2$,
and if we view $W$ as subspace of $H^0(Z,\,K_Z+D)$ as in \eqref{eqn-2-1}, then $W \cap \im(\varphi)=0$.
Therefore, \eqref{eqn-pre-1} follows immediately.
\end{proof}


\subsection{Proof of \autoref{thm-3-1}} \label{sec-pf-thm-3-1}
We follow the notations introduced in the last subsection.
According to \cite[Lemma\,2.2]{lu-zuo-15-0} together with the assumption, we have
\begin{equation}\label{eqn-low-gamma_i}
\gamma_i\geq g/4, \qquad \text{~if~} \deg(\iota_i)=2.
\end{equation}
If $\iota_{n-1}$ is not birational, neither is $\iota_i$ for $1\leq i\leq n-1$ by \eqref{eqn-factor-ij}.
Hence by \eqref{eqn-d_i-non-bi} and \eqref{eqn-low-gamma_i}, one has
$$d_i\geq \min\Big\{3(r_i-1),~2(r_i-1)+\frac{g}{2}\Big\},\qquad \forall~1\leq i\leq n-1.$$
In particular, taking $i=n-1$ one obtains $\delta \geq (g-1)/3$.
Hence \eqref{eqn-3-thm} follows the above inequalities together with
\eqref{eqn-degree-chi_f} and \eqref{eqn-xiao-1}.
Thus we may assume $\iota_{n-1}$ is birational in following. Let
$$\left\{\begin{aligned}
&x=\frac{2g-7\delta}{4g-11\delta}, &\quad&\lambda_0=\frac{16-5x}{3}=\frac{18g-47\delta}{4g-11\delta},&\qquad&\text{if~}\delta\leq \frac{2g}{21};\\[1mm]
&x=\frac{8g-14\delta}{16g-13\delta}, &~&\lambda_0=\frac{16-5x}{3}=\frac{72g-46\delta}{16g-13\delta},&&\text{if~}\delta \geq \frac{2g}{21}.
\end{aligned}\right.$$
According to \eqref{eqn-xiao-1} together with \eqref{eqn-second-mult}, one obtains
\begin{equation}\label{eqn-3-1}
\begin{aligned}
\omega_f^2 &~\geq~ \sum_{i=1}^{\tilde l-1}\big(x(2\theta_i-r_i)+(1-x)(d_i+d_{i+1})\big)\big(\mu_{i}-\mu_{i+1}\big)\\
&\quad+\sum_{i=\tilde l}^{n-1}\big(x(2\theta_i-r_i+\tilde\theta_i)+(1-x)(d_i+d_{i+1})\big)\big(\mu_{i}-\mu_{i+1}\big).
\end{aligned}
\end{equation}
We claim that
\begin{eqnarray}
x(2\theta_i-r_i)+(1-x)(d_i+d_{i+1}) &\geq& \lambda_0\cdot r_i-4,\label{eqn-3-2}\\
x(2\theta_i-r_i+\tilde\theta_i)+(1-x)(d_i+d_{i+1}) &\geq& \lambda_0\cdot r_i-4.\label{eqn-3-3}
\end{eqnarray}
Assume the above claim. Then \eqref{eqn-3-thm} follows directly from
\eqref{eqn-3-1} and \eqref{eqn-konno}.
Hence it suffices to prove \eqref{eqn-3-2} and \eqref{eqn-3-3}.

Consider first the case when $1\leq i \leq \tilde l-1$, and we divide the proof of \eqref{eqn-3-2} into several subcases
(keep \eqref{eqn-factor-ij} in mind).
\begin{list}{}
{\setlength{\labelwidth}{0.2cm}
\setlength{\leftmargin}{0.3cm}}
\item[$\bullet$] $\deg(\iota_i)\geq 4$. In this case, one can show \eqref{eqn-3-2} easily by using
\eqref{eqn-d_i-non-bi} and the definition of $\theta_i$ in \eqref{eqn-def-theta_i}.

\item[$\bullet$] $\deg(\iota_i)=3$.
According to \eqref{eqn-d_i-non-bi} and \eqref{eqn-def-theta_i},
one obtains $d_{i+1}\geq d_i\geq 3(r_i-1)$ and $\theta_i\geq 2r_i-1$.
Hence
\begin{eqnarray*}
&&x(2\theta_i-r_i)+(1-x)(d_i+d_{i+1})\\
&\geq& x(3r_i-2)+(1-x)(6r_i-6)\\
&=&(6-3x)r_i-(6-4x) \geq \lambda_0\cdot r_i-4, \qquad \text{if~}r_i\geq 3.
\end{eqnarray*}
If $\deg(\iota_{i+1})=3$, then $d_{i+1}\geq 3(r_{i+1}-1)\geq 3r_i$ by \eqref{eqn-d_i-non-bi},
from which \eqref{eqn-3-2} follows immediately.
If $\deg(\iota_{i+1})=1$, we have better bound for $d_{i+1}$ by \eqref{eqn-d_i-bi},
from which one can also show \eqref{eqn-3-2} when $r_{i}\leq 2$.

\item[$\bullet$] $\deg(\iota_i)=2$. We have two possibilities to deal with.
If $\gamma_i\geq r_i-1$, then
$$\theta_i=3r_i-3,\qquad d_{i+1}\geq d_i\geq 4(r_i-1),$$
from which one can show \eqref{eqn-3-2} easily.
If $\gamma_i\leq r_i-2$, then
$$\theta_i=2r_i+\gamma_i-1, \quad\text{and}\quad  d_{i+1}\geq d_i\geq 2(r_i+\gamma_i-1).$$
Note that $d_i\leq 2g-2\leq 8\gamma_i-2$, from which it follows that $\gamma_i\geq r_i/3$.
Hence
\begin{eqnarray*}
&&x(2\theta_i-r_i)+(1-x)(d_i+d_{i+1}) \\
&\geq& x(3r_i+2\gamma_i-2)+(1-x)(4r_i+4\gamma_i-4)\\
&=&(4-x)r_i+(4-2x)\gamma_i-(4-2x)\geq\lambda_0\cdot r_i-4.
\end{eqnarray*}

\item[$\bullet$] $\deg(\iota_i)=1$. In this case, the maps $\iota_i$ and $\iota_{i+1}$ are both birational.
Hence $\theta_i=3r_i-3$.
According to \eqref{eqn-d_i-bi}, one obtains
\begin{equation}\label{eqn-3-8}
d_i+d_{i+1}\geq\left\{\begin{aligned}
&3(r_i+r_{i+1})-10,&\quad&\text{if~}r_{i+1}<\frac{g+6}{3};\\
&3r_i-5+\frac{g}{2}+\frac{3r_{i+1}}{2}-2,&&\text{if~}r_{i+1}\geq  \frac{g+6}{3}\text{~and~}r_{i}< \frac{g+6}{3};\\
&g+\frac{3(r_i+r_{i+1})}{2}-4,&&\text{if~}r_{i}\geq \frac{g+6}{3}.
\end{aligned}\right.
\end{equation}
We only show \eqref{eqn-3-2} in the last possibility, and leave the proof of \eqref{eqn-3-2}
in the first two possibilities to the readers.
By \eqref{eqn-3-8}, one has $d_i+d_{i+1}\geq g+3r_i-2$ in this case.
By the definition of $\tilde l$ in \eqref{eqn-def-tilde-l},
one has $r_i+g\leq 2r_{n-1}-1=2(g-\delta)-1$, i.e., $g\geq r_i+2\delta+1$.
Note also that
$$2(1-x)\delta\geq (\lambda_0-x-4)(g-2\delta)\geq (\lambda_0-x-4)(r_i+1).$$
Thus
\begin{eqnarray*}
&&x(2\theta_i-r_i)+(1-x)(d_i+d_{i+1})\\
&\geq& x(5r_i-6)+(1-x)(4r_i+2\delta-1)\\
&=&(4+x)r_i+2(1-x)\delta-(1+5x)\\
&\geq &\lambda_0\cdot r_i-(5+6x-\lambda_0)>\lambda_0\cdot r_i-4.
\end{eqnarray*}
Therefore, \eqref{eqn-3-2} is proved.
\end{list}

Now consider the case when $\tilde l\leq i \leq n-1$. By the definition of $\tilde l$ in \eqref{eqn-def-tilde-l},
we have $\iota_i$ is birational, $r_i\geq \frac{g}{3}+2$ and $r_i\geq 2r_{n-1}-g=g-2\delta$.
Hence $\theta_i=3r_i-3$, and $d_i+d_{i+1}\geq g+3r_i-2$ by \eqref{eqn-d_i-bi}.
By definition, one checks easily that
$$3x\delta+\frac{2-5x}{2} g\geq \frac{14-31x}{6} r_i,\quad \forall~g-\delta \geq r_i\geq g-2\delta.$$
Thus
\begin{eqnarray*}
&&x(2\theta_i-r_i+\tilde\theta_i)+(1-x)(d_i+d_{i+1})\\
&\geq& x\left(5r_i-6+\frac32(r_i-g+2\delta)\right)+(1-x)(g+3r_i-2)\\
&=&\Big(3+\frac72x\Big)r_i+\left(3x\delta+\frac{2-5x}{2} g\right)-(2+4x)\\
&\geq &\lambda_0\cdot r_i-4.
\end{eqnarray*}
Therefore, \eqref{eqn-3-3} is proved. The proof is complete.
\qed

\subsection{Proof of \autoref{thm-3-2}}\label{sec-pf-thm-3-2}
Since $\delta<g$,
it follows that $g\geq 4$ by our assumption.
We divide the proof into two cases according to the relation between $\delta$ and $g$.

{\vspace{1mm}\sc\noindent Case 1: $\delta\geq \frac{3g-1}{5}$}.
Let
$$i_0=\min\Big\{i~\big|~r_i>\frac{r_{n-1}}{2}\Big\}=\min\Big\{i~\big|~r_i\geq \frac{g-\delta+1}{2}\Big\}.$$

If $i_0=1$, then $d_{1}\geq 3\left(\frac{r_{n-1}+1}{2}-1\right)$
by \autoref{lemma-d_i-2} and \eqref{eqn-low-gamma_i} since $\frac{r_{n-1}+1}{2}\leq \frac{g+2}{5}$.
Hence according to \eqref{eqn-xiao}, we get
$$\omega_f^2\geq (2g-2+d_1)\cdot\mu_1 \geq \frac{2g-2+d_1}{g-\delta}\cdot \chi_f> \frac{4(g-1)}{g-\delta/2}\cdot \chi_f.$$

If $i_0\geq 2$, then $r_{i_0-1}\leq \frac{r_{n-1}}{2}$, and $r_{i_0-1}\leq \frac{r_{n-1}-1}{2}$ when $r_{i_0}= \frac{r_{n-1}+1}{2}$.
Combining these with \autoref{lemma-d_i-2} and \eqref{eqn-low-gamma_i}, it is easy to show that
$$d_{i_0}\cdot (r_{n-1}-r_{i_0-1})\geq\left\{\begin{aligned}
&3,&\quad&\text{if~}g-\delta=2,\\
&\frac{3}{4}\big((g-\delta)^2-1\big),&&\text{if~}g-\delta\geq 3.
\end{aligned}\right.$$
Note that $g-\delta\geq 3$ implies $g\geq7$ by the assumption $\delta\geq \frac{3g-1}{5}$.
Therefore, according to \eqref{eqn-cor-xiao} we get
$$\begin{aligned}
\lambda_f &> \frac{(2g-2)^2}{(2g-2)\cdot r_{n-1}-d_{i_0}\cdot (r_{n-1}-r_{i_0-1})}\\
&\geq \left\{\begin{aligned}
&\frac{(2g-2)^2}{(2g-2)\cdot 2-3}\geq \frac{4(g-1)}{g-\delta/2}, &\quad&\text{if~}g-\delta=2;\\
&\frac{(2g-2)^2}{(2g-2)\cdot (g-\delta)-\frac{3}{4}\big((g-\delta)^2-1\big)}\geq \frac{4(g-1)}{g-\delta/2},&&\text{if~}g-\delta\geq 3.
\end{aligned}\right.
\end{aligned}$$

{\vspace{1mm}\sc\noindent Case 2: $\frac{3g-2}{5}\geq \delta \geq \frac{2(g+8)}{9}$}.
In this case, we have $g\geq 8$ since $\delta$ is an integer.

\begin{list}{}
{\setlength{\labelwidth}{0.2cm}
\setlength{\leftmargin}{0.3cm}}
\item[$\bullet$]
{\sc Subcase 2.1:} $\frac{3g-2}{5}\geq \delta \geq \frac{2g+2}{5}$.~
Let
$$i_1=\min\big\{i~\big|~d_i\geq g-1\big\}.$$
Then according to \eqref{eqn-xiao}, one has
\begin{equation}\label{eqn-3-6}
\begin{aligned}
\omega_f^2\geq~& \sum_{i=1}^{i_1-1}\big(d_{i}+d_{i+1}\big)\big(\mu_{i}
-\mu_{i+1}\big)+(2g-2+d_{i_1})\mu_{i_1}\\
=~& \sum_{i=1}^{i_1-1}\big(d_{i}+d_{i+1}\big)\big(\mu_{i}-\mu_{i+1}\big)
+\sum_{i=i_1}^{n-1}\big(2g-2+d_{i_1}\big)\big(\mu_{i}-\mu_{i+1}\big).
\end{aligned}
\end{equation}
We claim that
\begin{eqnarray}
d_{i}+d_{i+1} &\geq& \frac{2(g-1)}{g-\delta/2-1}\cdot (2r_i-1),\qquad \forall~1\leq i \leq i_1-1;\label{eqn-3-7-1}\\
2g-2+d_{i_1} &\geq& \frac{2(g-1)}{g-\delta/2-1}\cdot (2r_i-1), \qquad \forall~i_1\leq i \leq n-1.\label{eqn-3-7-2}
\end{eqnarray}
Assuming the above claim, one obtains from \eqref{eqn-3-6} together with \eqref{eqn-degree-chi_f} that
$$\omega_f^2\geq \frac{4(g-1)}{g-\delta/2-1}\cdot\chi_f-\frac{2(g-1)}{g-\delta/2-1}\mu_1.$$
Combining this with \eqref{eqn-konno}, we prove \eqref{eqn-thm-3-2} in this subcase.

It remains to show \eqref{eqn-3-7-1} and \eqref{eqn-3-7-2}.
Since $d_{i_1}\geq g-1$, \eqref{eqn-3-7-2} follows immediately since $r_i\leq r_{n-1}=g-\delta$.
Note also that $\frac{2(g-1)}{g-\delta/2-1}\leq 3$ by our assumption,
and $d_i\geq 3(r_i-1)$ for $1\leq i\leq i_1-1$ by \autoref{lem-2-1}.
Hence \eqref{eqn-3-7-1} follows for $i\leq i_1-2$.
When $i=i_1-1$, by \autoref{rem-2-2}, we have
either $d_{i_1-1}+d_{i_1}\geq 3(2r_{i_1-1}-1)$, or $d_{i_1-1}+d_{i_1}=6r_{i_1-1}-4$ and $r_{i_1-1}\in\big\{g/3,\,(g+1)/3\big\}$.
Since $g\geq 8$, one can also verify \eqref{eqn-3-7-1} for $i=i_1-1$,
except when $g=9$, $\delta=5$, $d_{i_1}=8$, $d_{i_1-1}=6$ and $r_{i_1-1}=3$.
For the exceptional case, we replace $i_1$ by $i_1-1$ in \eqref{eqn-3-6}. Then one can show easily that
both \eqref{eqn-3-7-1} and \eqref{eqn-3-7-2} hold, and hence proves \eqref{eqn-thm-3-2}.

\item[$\bullet$]
{\sc Subcase 2.2:} $\frac{2g+1}{5}\geq \delta \geq \frac{2(g+9)}{9}$,
or $\delta=\frac{2g+17}{9}$ or $\frac{2g+16}{9}$ and $g\leq 52$.~
Let
$$\begin{aligned}
x&\,=\frac{2(g-1)}{g-\delta/2-t}~\text{~with~}t=\frac{17}{18},\\[1mm]
i_1&\,=\min\left\{i~\big|~d_i\geq g-1\right\},\\[1mm]
i_2&\,=\min\left\{i~\big|~d_i\geq x\big(g-3\delta/2-(1-t)\big)\right\}.
\end{aligned}$$
Note that $9/4<x<3$ and $i_1\leq i_2$ by our assumption.

If $i_1=i_2$, then we can show similarly as the above subcase that
\begin{eqnarray*}
d_{i}+d_{i+1} &\geq& x (2r_i-1),\qquad \forall~1\leq i \leq i_1-1;\\
2g-2+d_{i_1} &\geq& x (2r_i-1), \qquad \forall~i_1\leq i \leq n-1.
\end{eqnarray*}
Hence \eqref{eqn-thm-3-2} follows from \eqref{eqn-3-6} together with \eqref{eqn-konno}.

In the rest part of the proof, we assume that $i_1<i_2$.
Before going further, we first claim that
\begin{claim}\label{claim-2-1}
{\rm (1).} If $d_i<x(g-3\delta/2)$, then $d_i\geq x(r_i-1)$.

{\rm (2).}
If $d_i<x(g-3\delta/2)-\frac12-\frac{5(2g+1-5\delta)}{8(2g-\delta-t)}$, then $r_i<g-(3\delta-1)/2$.
\end{claim}
\begin{proof}[Proof of \autoref{claim-2-1}]
(1).
Let $\iota_i$ be defined as in \eqref{eqn-def-iota_i}.
Since $x\leq 3$ by assumption,
the claim follows immediately if $\deg(\iota_i)\geq 3$ by \eqref{eqn-d_i-non-bi}.
When $\deg(\iota_i)\leq 2$, we prove the claim by contradiction.
Assume that
\begin{equation}\label{eqn-2-4}
d_i< x (r_i-1).
\end{equation}

Consider first the case when $\deg(\iota_i)=2$.
By \eqref{eqn-d_i-non-bi} together with \eqref{eqn-2-4},
we may assume that $r_i-1>\gamma_i$, and hence $d_i\geq 2(r_i-1)+2\gamma_i\geq 2(r_i-1)+\frac{g}{2}$.
Combining this with \eqref{eqn-2-4}, we get
$$\frac{g}{2}<(x-2)(r_i-1)\leq (x-2)(g-\delta-1)<\frac{(g-\delta-1)\delta}{g-\delta/2-t},
\quad\text{which is a contradiction}.$$

We now consider the case when $\deg(\iota_i)=1$, i.e., $\iota_i$ is birational.
Hence $r_i\geq 3$. Moreover, if $r_i=3$, then
$8\leq g\leq \frac{(d_i-1)(d_i-2)}{2}$, which implies that
$d_i\geq 6> x(r_i-1)$.
Hence we may assume that $r_i\geq 4$ in the following.
According to Castelnuovo's bound \eqref{eqn-castelnuovo},
one has $d_i\geq 4r_i-7\geq 3(r_i-1)\geq x(r_i-1)$ if $m_i\geq 4$.
It remains to consider the cases when $m_i=3$ or $2$.

When $m_i=3$, one has $d_i-1\geq 3(r_i-2)$, i.e., $d_i\geq 3r_i-5$.
Since $x< 3$ by assumption,
it suffices to consider the cases when $d_i=3r_i-5$ or $3r_i-4$.
By Castelnuovo's bound \eqref{eqn-castelnuovo}, we have
\begin{equation}\label{eqn-3-9}
d_i\geq \frac{g}{3}+2r_i-3.
\end{equation}
\mbox{\,}~If $d_i=3r_i-5$, then $r_i-1\geq \frac{g}{3}+1$ by \eqref{eqn-3-9},
and $2>(3-x)(r_i-1)$ by \eqref{eqn-2-4}.
Hence $$\delta>\frac{2g-6}{3}+\frac{(2-2t)(g+1)}{(g-1)}>\frac{2g-6}{3},\quad\text{which contradicts the assumption.}$$
\mbox{\,}~If $d_i=3r_i-4$, then $r_i-1\geq \frac{g}{3}$ by \eqref{eqn-3-9},
and $1>(3-x)(r_i-1)$ by \eqref{eqn-2-4}.
Hence $$\delta>\frac{2g-6}{3}+2(1-t)>\frac{2g-6}{3},\quad\text{which is still a contradiction.}$$

When $m_i=2$, one has $d_i\geq \frac{g-1}{2}+\frac{3(r_i-1)}{2}$
by Castelnuovo's bound \eqref{eqn-castelnuovo}.
Combining this with \eqref{eqn-2-4} and the assumption $d_i<x(g-3\delta/2)$ respectively, we obtain
\begin{equation*}
\left\{\begin{aligned}
r_i-1&\,>\frac{(g-1)(2g-\delta-2t)}{2g+3\delta+6t-8};\\
r_i-1&\,<\frac{(g-1)(6g-11\delta+2t)}{3(2g-\delta-6t)}.
\end{aligned}\right.
\end{equation*}
Hence
$$\begin{aligned}
&\quad \frac{(g-1)(2g-\delta-2t)}{2g+3\delta+6t-8}<\frac{(g-1)(6g-11\delta+2t)}{3(2g-\delta-6t)},
~\,\Longrightarrow~\\[1mm]
&\,0\,<\,\delta(2g+5-9\delta)+\frac{26g-34}{9}\\[1mm]
& \left\{\begin{aligned}
\leq\,&\frac{2g+18}{9}\cdot(2g+5-2g-18)+\frac{26g-34}{9}<0,
\quad\text{~if~}\delta\geq \frac{2g+18}{9};\\[1mm]
=\,&\frac{(36-2\ell)g-\ell(\ell-5)-34}{9}<0,
\quad\text{~if~}\delta=\frac{2g+\ell}{9}\text{~with~}16\leq \ell<18\text{~and~}g\leq 52.
\end{aligned}\right.
\end{aligned}
$$
The above contradiction completes the proof.

(2).
By (1), one has $r_i-1<g-3\delta/2$.
Hence it suffice to derive a contradiction if $r_i=g-(3\delta-1)/2$.
The proof is similar as above.
In fact, one can easily prove a contradiction except the case
when $\deg(\iota_i)=1$ and $m_i=2$.
In the exceptional case,
$\delta\geq \frac{2g+17}{9}$ since $\delta$ is odd,
and by Castelnuovo's bound \eqref{eqn-castelnuovo} we obtain
$$x\Big(g-\frac{3\delta}{2}\Big)-\frac12-\frac{5(2g+1-5\delta)}{8(2g-\delta-t)}>d_i
\geq \frac{g-1}{2}+\frac{3(r_i-1)}{2}=2g-\frac{9\delta+5}{4}.$$
Hence
$$0>\delta(18\delta-4g-33)-\frac29g+\frac{49}{4}.$$
This is a contradiction since $\delta\geq \frac{2g+17}{9}$.
\end{proof}

We now come back to the proof of \eqref{eqn-thm-3-2}.
By \autoref{lem-2-1} and \autoref{rem-2-2}, one has
\begin{equation}\label{eqn-3-11}
\left\{\begin{aligned}
d_{i}\,+\,d_{i+1}~&\,\geq 6r_{i}-3\geq 2xr_i-(2x-3),&\quad&\text{if~}i<i_1-1;\\
d_{i_1-1}+d_{i_1}&\,\geq 6r_{i_1-1}-3\geq 2xr_{i_1-1}-(2x-3),&&\text{if~}d_{i_1-1}<g-3.
\end{aligned}\right.
\end{equation}
By \autoref{claim-2-1}, we have
\begin{equation}\label{eqn-3-12}
\left\{\begin{aligned}
d_{i}\,+\,d_{i+1}~&\,\geq 2xr_{i}-x,&\quad&\text{if~}i<i_2-1;\\
d_{i_2-1}+d_{i_2}&\,\geq 2xr_{i_2-1}-x,&&\text{if~}d_{i_2-1}<\Delta.
\end{aligned}\right.
\end{equation}
Here $\Delta\triangleq x(g-3\delta/2)-\frac12-\frac{5(2g+1-5\delta)}{8(2g-\delta-t)}$.
If $d_{i_2-1}\geq \Delta$, then
$r_{i_2-1}=g-(3\delta-1)/2$ by \autoref{claim-2-1}\,(1),
and hence
\begin{equation}\label{eqn-3-13}
d_{i_2-1}+d_{i_2}\geq 2d_{i_2-1}+1\geq 2xr_{i_2-1}-x-\frac{5(2g+1-5\delta)}{4(2g-\delta-t)}.
\end{equation}
Note also that $2g-2+d_{i_2}\geq x\big(2(g-\delta)-1\big)$.
Hence by \eqref{eqn-xiao} and \eqref{eqn-degree-chi_f}, one has
\begin{equation*}
\begin{aligned}
\omega_f^2\geq& \sum_{i=1}^{i_2-1}\big(d_{i}+d_{i+1}\big)\big(\mu_{i}
-\mu_{i+1}\big)+(2g-2+d_{i_2})\mu_{i_2}\\
\geq&\left\{\begin{aligned}
&2x\chi_f-(2x-3)\mu_1-(3-x)\mu_{i_1},&&\text{if $d_{i_1-1}<g-3$ and $d_{i_2-1}<\Delta$};\\[0.5mm]
&2x\chi_f-(2x-3)\mu_1-(3-x)\mu_{i_1-1},&&\text{if $d_{i_1-1}\geq g-3$ and $d_{i_2-1}<\Delta$};\\[0.5mm]
&2x\chi_f-(2x-3)\mu_1-(3-x)\mu_{i_1}-\xi\mu_{i_2-1},&&\text{if $d_{i_1-1}<g-3$ and $d_{i_2-1}\geq \Delta$};\\[0.5mm]
&2x\chi_f-(2x-3)\mu_1-(3-x)\mu_{i_1-1}-\xi\mu_{i_2-1},&&\text{if $d_{i_1-1}\geq g-3$ and $d_{i_2-1}\geq \Delta$}.
\end{aligned}
\right.
\end{aligned}
\end{equation*}
Here $\xi=\frac{5(2g+1-5\delta)}{4(2g-\delta-t)}.$
By \eqref{eqn-xiao}, we also have
$$\omega_f^2\geq (d_1+d_i)(\mu_1-\mu_i)+(2g-2+d_i)\mu_i\geq d_i\mu_1+(2g-2)\mu_i,
\quad\forall~1\leq i\leq n-1.$$
Hence
\begin{equation*}
\lambda_f\geq \Lambda\triangleq\left\{\begin{aligned}
 &\frac{4(g-1)x}{2g-2+2x-3+(3-x)\big(1-\frac12\big)},
&&\text{if $d_{i_1-1}<g-3$ and $d_{i_2-1}<\Delta$};\\[0.5mm]
 &\frac{4(g-1)x}{2g-2+2x-3+(3-x)\big(1-\frac{g-3}{2g-2}\big)},
&&\text{if $d_{i_1-1}\geq g-3$ and $d_{i_2-1}<\Delta$};\\[0.5mm]
 &\frac{4(g-1)x}{2g+2x-5+\frac{3-x}{2}+\frac{(2g-2-\Delta)\xi}{2g-2}},
&&\text{if $d_{i_1-1}<g-3$ and $d_{i_2-1}\geq \Delta$};\\[0.5mm]
 &\frac{4(g-1)x}{2g+2x-5+\frac{(g+1)(3-x)}{2g-2}+\frac{(2g-2-\Delta)\xi}{2g-2}},
&&\text{if $d_{i_1-1}\geq g-3$ and $d_{i_2-1}\geq \Delta$}.
\end{aligned}
\right.
\end{equation*}
Note that $\Delta>g-1$.
Thus one shows that $\Lambda>\frac{4(g-1)}{g-\delta/2}$.
This proves \eqref{eqn-thm-3-2} in this subcase.

\item[$\bullet$]
{\sc Subcase 2.3:} $\delta=\frac{2g+17}{9}$ or $\frac{2g+16}{9}$ and $g>52$.~
In this subcase,
\eqref{eqn-thm-3-2} follows directly from \eqref{eqn-3-thm}.\qedhere
\end{list}
This completes the proof.
\qed

\section{Double cover fibrations}\label{sec-double}
In this section, we treat the double cover fibrations.
So we always assume in the section that $f:\,X\to B$ is a locally non-trivial double cover fibration of type $(g, \gamma)$ as in \autoref{def-double-cover}.
Since the case where $\gamma=0$ has been studied in \cite{xiao-91,lu-zuo-13} (see also \cite{cornalba-harris-88,lu-zuo-14} for the semi-stable case),
$\gamma$ is assumed to be positive in this section unless other explicit statements.

In \autoref{section-invarians-double-cover}, we prove the formulas for the invariants of the double cover fibrations.
In \autoref{sec-irr-double}, we consider the irregular double cover fibrations.
In \autoref{sec-slope-double}, we study the slope problems.
Finally, we prove \autoref{prop-2-21} (resp. \autoref{thm-3-4}) in \autoref{sec-pf-prop-2-21} (resp. \autoref{sec-pf-thm-3-4}).

\subsection{Invariants of double cover fibrations}\label{section-invarians-double-cover}
In this subsection, we first define the local invariants of the induced double cover,
and then show in \autoref{thminvariants-double-fibration} that the relative invariants
of $f$ can be expressed by these local invariants and relative invariants of the quotient fibration.

The degree-two morphism $\pi$ induces an involution $\sigma$ on $X$.
Let $\vartheta:\,\wt X \to X$ be the composition of all the blowing-ups of the isolated fixed points of $\sigma$, and $\tilde \sigma$ the induced involution on $\wt X$.
Then the quotient $\wt Y:=\wt X /\langle\tilde\sigma\rangle$ is a smooth surface with a natural fibration $\wt h:\,\wt Y \to B$ of genus $\gamma$, which may not be relatively minimal.
Let $h:\,Y \to B$ be its relatively minimal model.
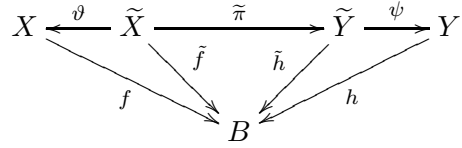
\begin{figure}[H]
$$
\xymatrix{
X \ar[drr]_-{f} & \wt X \ar[l]_-{\vartheta}\ar[dr]^-{\tilde f}\ar[rr]^-{\wt\pi}&&\wt Y \ar[dl]_-{\tilde h}\ar[r]^-{\psi} &Y \ar[dll]^-{h}\\
&&B&&
}$$
\caption{Double cover fibration.}  \label{figure-5-1}
\end{figure}

The double cover $\tilde \pi$ induces a double cover $\pi_0:\,X_0 \to Y_0:=Y$, which is determined by the relation
$\mathcal O_{Y}(R) \equiv L^{\otimes 2}$ with $R=\psi(\wt R)$ and $\wt R$ being the branch locus of $\tilde\pi$.
According to Hurwitz formula, one has
\begin{equation}\label{eqnhurwitz-formula}
R\cdot \Gamma=2g+2-4\gamma\geq 0,\qquad ~\text{for any fiber~$\Gamma$ of $h$}.
\end{equation}
The surface $X_0$ is normal but not necessarily smooth.
Moreover, $\tilde\pi$ is in fact the canonical resolution of $\pi_0$ (cf. \cite[\S\,III.7]{bhpv-04}):
\begin{figure}[H]
$$\mbox{}
  \xymatrix{
\wt X\ar@{=}[r] & X_{t} \ar[r]^-{\phi_t}\ar[d]_-{\tilde \pi=\pi_t}&
 X_{t-1}\ar[r]^-{\phi_{t-1}}\ar[d]^-{\pi_{t-1}}&
\cdots \ar[r]^-{\phi_2} & X_1\ar[r]^-{\phi_1}\ar[d]_-{\pi_1}
& X_0 \ar[d]^-{\pi_0}\\
\wt Y\ar@{=}[r] & Y_{t} \ar[r]^-{\psi_t}& Y_{t-1}\ar[r]^-{\psi_{t-1}}&
\cdots \ar[r]^-{\psi_2} & Y_1\ar[r]^-{\psi_1} & Y_0 \ar@{=}[r] & Y}
$$
\caption{Canonical resolution.}  \label{figure-1}
\end{figure}

\noindent Here $\psi_i$'s are successive blowing-ups
resolving the singularities of $R$, and
$\pi_{i}:\,X_i \to Y_i$ is the double cover determined by
$\mathcal O_{Y_i}(R_i) \equiv L_i^{\otimes 2}$ with
$$R_i=\psi_i^*(R_{i-1})-2[m_{i-1}/2]\, \mathcal E_i,\qquad
L_i=\psi_i^*(L_{i-1})\otimes \mathcal O_{Y_i}\left(\mathcal E_i^{-[m_{i-1}/2]}\right),$$
where $\mathcal E_i$ is the exceptional divisor of $\psi_i$,
$m_{i-1}$ is the multiplicity of the singular point $y_{i-1}$ in $R_{i-1}$ (also called the multiplicity of the blowing-up $\psi_i$),
$[~]$ stands for the integral part,
$R_0=R$ and $L_0=L$.
A singularity $y_j \in R_{j}\subseteq Y_{j}$ is said to be {\it infinitely closed to}
$y_{i}\in R_{i}\subseteq Y_{i}$ ($j>i$), if $\psi_{i+1}\circ\cdots\circ\psi_j(y_j)=y_{i}\,.$

We remark that the order of these blowing-ups contained in $\psi$ is not unique.
If $y_{i-1}$ is a singular point of $R_{i-1}$ of odd multiplicity $2k+1$ ($k\geq 1$)
and there is a unique singular point $y$ of $R_i$
on the exceptional curve $\mathcal E_i$ of multiplicity $2k+2$,
then we always assume that $\psi_{i+1}: Y_{i+1} \to Y_{i}$ is a blowing-up at $y_i=y$.
We call such a pair $(y_{i-1},y_{i})$ a {\it singularity of $R$ of type $(2k+1 \to 2k+1)$},
and $y_{i-1}$ (resp. $y_i$) the first (resp. second) component.
\begin{definition}\label{definitionofs_i}
For any singular fiber $F$ of $f$ and $j\geq 2$, we define
\begin{list}{}
{\setlength{\labelwidth}{0.3cm}
\setlength{\leftmargin}{0.4cm}}
\item[$\bullet$] if $j$ is odd, $s_j(F)$ equals the number of $(j\to j)$
                 type singularities of $R$ over the image $f(F)$;
\item[$\bullet$] if $j$ is even, $s_j(F)$ equals the number of singularities of multiplicity $j$ or $j+1$ of $R$ over the image $f(F)$,
                 neither belonging to the second component of type $(j-1 \to j-1)$ singularities nor to the first component of type $(j+1 \to j+1)$ singularities.
\end{list}
Let $\omega_{\tilde h}=\omega_{\wt Y}\otimes\tilde h^*\omega_{B}^{-1}$ and $\wt R'=\wt R \setminus\wt V$,
where $\wt V$ is the union of vertical isolated $(-2)$-curves in $\wt R$.
Here a curve $C\subseteq \wt R$ is called to be {\it isolated} in $\wt R$,
if there is no other curve $C'\subseteq \wt R$ such that $C\cap C' \neq \emptyset$. We define
i$$\begin{aligned}
s_2&:=\left(\omega_{\tilde h}+\wt R'\right)\cdot \wt R'+2\sum_{F \text{~is singular}} s_2(F),\\
s_j&:=\sum_{F \text{~is singular}} s_j(F),\qquad\qquad \forall~ j\geq 3.
\end{aligned}$$
Note that the contraction $\psi$ is unique since $\gamma>0$ (although the order of these blowing-ups contained in $\psi$ is not unique).
Hence the invariants $s_j$'s are well-defined.
By definition, $s_j$ is non-negative for $j\geq 3$, but it is not clear whether $s_2$ is non-negative or not.
\end{definition}

\begin{lemma}\label{numberofcontraction}
Let $F$ be a singular fiber of the fibration $f$,
and $\wt F$ (resp. $\wt \Gamma$, resp. $\Gamma$)
the corresponding fiber in $\wt X$ (resp. $\wt Y$, resp. $Y$).
Then the $(-1)$-curves in $\wt F$ are in one-to-one correspondence to the isolated $(-2)$-curves
of $\wt R$, which are also contained in $\wt \Gamma$.
And the number of these $(-1)$-curves is equal to
$$n_2(F)+\sum_{k\geq 1} s_{2k+1}(F),$$
where $n_2(F)$ is the number of isolated $(-2)$-curves of $R$, which are also contained in $\Gamma$.
\end{lemma}
\begin{proof}
Note that the $(-1)$-curves in $\wt F$ are exactly the inverse image of the isolated fixed points of $\sigma$ on $F$, hence fixed by $\tilde\sigma$.
It follows that these $(-1)$-curves in $\wt F$ are in one-to-one correspondence to the isolated $(-2)$-curves
of $\wt R$, which are also contained in $\wt \Gamma$.

Let $E$ be such a $(-2)$-curve of $\wt R$. Then it is the strict inverse image of either an exceptional curve $\mathcal E_i$ or an irreducible curve $C$ on $\Gamma$.
In the first case, it is easy to see that $y_{i-1}=\psi_i(\mathcal E_i)$ is a singularity of $R_{i-1}$ with odd multiplicity $2k+1$,
and that $R_i$ has a unique singularity on $\mathcal E_i$ with multiplicity $2k+2$.
Equivalently, it corresponds to a singularity of $R$ of type $(2k+1 \to 2k+1)$.
In the later case, let
$$E=\psi^*(C)-\sum a_j\mathcal E_j,\quad \text{with~}a_j\geq 0.$$
Then
$$-2=E^2=C^2-\sum a_j^2,\qquad 0=\omega_{\wt Y}\cdot E=\omega_{Y}\cdot C+\sum a_j.$$
On the other hand, one has $C^2\leq 0$ and $C^2=0$ if and only if $\Gamma=nC$ for some $n$, since $C\subseteq \Gamma$.
Hence it follows that $C^2\neq 0$ since $\gamma>0$,
and that $C^2\neq -1$; otherwise by construction $C$ must be smooth and hence is $(-1)$-curve, which is impossible due to the relative minimality of $h$.
Therefore, $C$ must be an isolated $(-2)$-curve of $R$, which is also contained in $\Gamma$.

Conversely, it is clear that each singularity of $R$ of type $(2k+1\to 2k+1)$ creates an isolated $(-2)$-curve contained in $\wt R$, and that the inverse image of each isolated $(-2)$-curve
in $R$ is still an isolated $(-2)$-curve in $\wt R$.
The proof is complete.
\end{proof}

\begin{theorem}\label{thminvariants-double-fibration}
Let $f$ be a double cover fibration of type $(g, \gamma)$, and $s_i$'s the singularity indices as above. Then
$$\begin{aligned}
(2g+1-3\gamma)\omega_f^2~=&~\,x\cdot\frac{\omega_h^2}{\gamma-1}+yT+zs_2+\sum_{k\geq 1} a_ks_{2k+1}+\sum_{k\geq 2}b_ks_{2k},\\
(2g+1-3\gamma)\chi_f~=&~\,\bar x\cdot\frac{\omega_h^2}{\gamma-1}+2(2g+1-3\gamma)\chi_h+\bar yT\\
                    &\,+\bar zs_2-\frac{2g+1-3\gamma}{4}\cdot n_2 +\sum_{k\geq 1} \bar a_ks_{2k+1}+\sum_{k\geq 2}\bar b_ks_{2k},\\
e_f~=&~\,2e_h+s_2-3n_2+\sum_{k\geq 1} s_{2k+1}+\sum_{k\geq 2}2s_{2k},
\end{aligned}$$
where we set $\frac{\omega_h^2}{\gamma-1}=0$ if $\gamma=1$,~
$n_2=\sum\limits_{F~\text{is singular}} n_2(F),$ and
$$\begin{aligned}
&x=\frac{(3g+1-4\gamma)(g-1)}{2},&\,\quad\quad\quad\,&~ y=\frac{3}{2},&\qquad&~z=g-1;\qquad\quad\\
&\bar x=\frac{(g+1-2\gamma)^2}{8},&~&~ \bar y=\frac{1}{8},&&~\bar z=\frac{g-\gamma}{4}.
\end{aligned}$$\vspace{-0.2cm}
$$\begin{aligned}
&a_k\,=\,12\bar a_k-(2g+1-3\gamma),
&&b_k\,=\,12\bar b_k-2(2g+1-3\gamma),\\
&\bar a_k\,=\,k\big(g-1+(k-1)(\gamma-1)\big),
&\quad&\bar b_k\,=\,\frac{k\big(g-1+(k-2)(\gamma-1)\big)}{2},~
\end{aligned}~$$\vspace{-0.2cm}
$$~T=-\frac{\big((g+1-2\gamma)\omega_h-(\gamma-1)R\big)^2}{\gamma-1}-2(\gamma-1)n_2\geq 0.\qquad\qquad\qquad
$$
\end{theorem}
\begin{proof}
Recall the canonical resolution $\psi$ exhibited in \autoref{figure-1}.
By \autoref{numberofcontraction}, one has
$$\begin{aligned}
&\left(\omega_{\tilde h}+\wt R'\right)\cdot \wt R'-2\left(n_2+\sum_{k\geq 1} s_{2k+1}\right)\\=\,&\left(\omega_{\tilde h}+\wt R\right)\cdot \wt R
=\left(\omega_{h}+R\right)\cdot R-\sum_{i=1}^t\left(\left[\frac{m_i}2\right]-1\right)\cdot\left[\frac{m_i}2\right]\\
=\,&\left(\omega_{h}+R\right)\cdot R-\sum_{k\geq1}(8k^2+4k+2)s_{2k+1}-\sum_{k\geq 2}(4k^2-2k)s_{2k}-2\sum_{F~\text{is singular}} s_2(F).
\end{aligned}$$
Combining this with the definition of $s_2$, we get
\begin{equation}\label{eqn(omegah+R)R}
(\omega_{h}+R)\cdot R=(s_2-2n_2)+\sum_{k\geq1}4k(2k+1)s_{2k+1}+\sum_{k\geq 2}2k(2k-1)s_{2k}.
\end{equation}
Thus by the formulas for double covers (cf. \cite[\S\,V.22]{bhpv-04}), one obtains:
\begin{eqnarray}
\hspace{-0.2cm}\omega_{\tilde f}^2&\hspace{-0.2cm}=\hspace{-0.2cm}&2\left(\omega_h^2+\omega_h\cdot R +\frac{R^2}{4}\right)-2\left(\sum_{k\geq1}(2k^2-2k+1)s_{2k+1}+\sum_{k\geq 2}(k-1)^2s_{2k}\right)\nonumber\\
&\hspace{-0.2cm}=\hspace{-0.2cm}&x'\cdot\frac{\omega_h^2}{\gamma-1}+y'\big(T+2(\gamma-1)n_2\big)+z'(\omega_{h}+R)\cdot R\label{eqn-omega-tilde-f^2}\\
&\hspace{-0.2cm}\hspace{-0.2cm}&-2\left(\sum_{k\geq1}(2k^2-2k+1)s_{2k+1}+\sum_{k\geq 2}(k-1)^2s_{2k}\right),\nonumber
\end{eqnarray}
\begin{eqnarray}
\hspace{-0.2cm}\chi_{\tilde f}&\hspace{-0.2cm}=\hspace{-0.2cm}&2\chi_{h}+\frac12\left(\frac{\omega_h\cdot R}{2}+\frac{R^2}{4}\right)
                               -\left(\sum_{k\geq1}k^2s_{2k+1}+\sum_{k\geq 2}\frac{k(k-1)}{2}s_{2k}\right)\qquad\qquad\,\nonumber\\
&\hspace{-0.2cm}=\hspace{-0.2cm}&2\chi_{h}+\bar x'\cdot\frac{\omega_h^2}{\gamma-1}+\bar y'\big(T+2(\gamma-1)n_2\big)+\bar z'(\omega_{h}+R)\cdot R\label{eqn-chi-tilde-f}\\
&\hspace{-0.2cm}\hspace{-0.2cm}&-\left(\sum_{k\geq1}k^2s_{2k+1}+\sum_{k\geq 2}\frac{k(k-1)}{2}s_{2k}\right),\nonumber
\end{eqnarray}
where $\ast'=\frac{\ast}{2g+1-3\gamma}$ for $\ast=x,y,z,\bar x,\bar y$ or $\bar z$.
Note that $\omega_f^2=\omega_{\tilde f}^2+n_2+\sum\limits_{k\geq 1} s_{2k+1}$
and $\chi_{f}=\chi_{\tilde f}$ by \autoref{numberofcontraction}.
Therefore, the formulas in our theorem follow from the above equalities together with \eqref{eqn(omegah+R)R} and \eqref{eqnnoether}.

Note that $T=2(g-1)\omega_h\cdot R \geq 0$ if $\gamma=1$.
It remains to show that $T\geq 0$ if $\gamma>1$. For this purpose, let $V\subseteq R$ be these isolated $(-2)$-curves contracted by $h$, and $R'=R\setminus V$.
By \autoref{numberofcontraction}, the number of components contained in $V$ is $n_2$.
Since $\Gamma\cdot \big((g+1-2\gamma)\omega_h-(\gamma-1)R'\big)=0$, one gets by Hodge index theorem that
$$0\geq \big((g+1-2\gamma)\omega_h-(\gamma-1)R'\big)^2=\big((g+1-2\gamma)\omega_h-(\gamma-1)R\big)^2+2(\gamma-1)^2n_2.$$
Hence $T\geq 0$ as required.
\end{proof}

\subsection{Irregular double cover fibrations}\label{sec-irr-double}
In this subsection, we would like to prove the following restrictions on the invariants of irregular double cover fibrations.
\begin{definition}\label{def-irr-double}
The double cover fibration $f$ is called irregular if the irregularity $q_{\pi}:=q(\wt X)-q(\wt Y)$
of the induced double cover $\pi$ is positive,
where $\wt X$ and $\wt Y$ are the same as in the last subsection.
\end{definition}
%

\begin{proposition}\label{prop-restriction-invariants}
Let $f:\,X\to B$ be a double cover fibration of type $(g, \gamma)$.

{\rm (i)}
If the double cover $\pi$ is irregular, i.e., $q_{\pi}>0$, then
\begin{eqnarray}
\hspace{-0,3cm}&&\hspace{-0,3cm}2(g+1-2\gamma)s_2 \label{eqnq_pi>0}\\
\hspace{-0,3cm}&\leq&\hspace{-0,3cm}(g+1-2\gamma)^2\cdot \frac{\omega_h^2}{\gamma-1}+ T+\sum_{k\geq 1} 2(4\bar a_k+2g+1-3\gamma)s_{2k+1}+\sum_{k\geq 2}8\bar b_ks_{2k}.\qquad\nonumber
\end{eqnarray}

{\rm (ii)}
If the image $J_0(\wt X)\subseteq \Alb_0(\wt X)$ is a curve of geometric genus $g'>0$, then
\begin{eqnarray}
\hspace{-0,3cm}&&\hspace{-0,3cm}2(g+1-2\gamma)\left(s_2+\sum_{k\geq 1}^{g'-1}  4(2k+1)ks_{2k+1}+\sum_{k\geq 2}^{g'} 2(2k-1)ks_{2k}\right) \label{eqnq_pi>g'}\\
\hspace{-0,3cm}&\leq&\hspace{-0,3cm}(g+1-2\gamma)^2\cdot \frac{\omega_h^2}{\gamma-1}+ T+\sum_{k\geq g'}  2(4\bar a_k+2g+1-3\gamma) s_{2k+1}+\sum_{k\geq g'+1} 8\bar b_ks_{2k};\qquad\nonumber
\end{eqnarray}
where $\bar a_k$'s, $\bar b_k$'s are defined in \autoref{thminvariants-double-fibration}, and $J_0$ will be defined in \eqref{eqndef-of-J_0}.
\end{proposition}

The main tool to prove the above proposition is the usage of Albanese varieties.
We first review the Albanese varieties and show that the ramified divisor is contracted by $J_0$.
Then the proposition follows from the semi-negativity of the divisors contracted by some non-trivial map.

Let $\wt {\mathcal R}=\tilde\pi^{-1}(\wt R)\subseteq \wt X$ the ramified divisor.
Let $\Alb(\wt X)$ (resp. $\Alb(\wt Y)$) be the Albanese variety of $\wt X$ (resp. $\wt Y$),
and $\tau$ the generator of the Galois group ${\rm Gal}(\wt X/\wt Y)\cong \mathbb Z/2\mathbb Z$.
Then we have a natural map $\Alb(\tilde\pi):\, \Alb(\wt X)\to \Alb(\wt Y)$ and $\tau$ has a natural action on $\Alb(\wt X)$.
Let
$$\Alb_0(\wt X)=\left\{x\in \Alb(\wt X)~\big|~\tau(x)=-x\right\}.$$
Then it is clear that $\Alb(\wt X)$ is isogenous to $\Alb_0(\wt X) \oplus \Alb(\tilde\pi)^{-1}\big(\Alb(\wt Y)\big)$ and $\dim \Alb_0(\wt X)=q_{\pi}$.
Denote by
\begin{equation}\label{eqndef-of-J_0}
J_0:\,\wt X \to \Alb_0(\wt X)
\end{equation}
the induced map.

\begin{lemma}\label{lemmawtR-contracted}
The ramified divisor $\wt {\mathcal R}$ is contracted by the map $J_0$.
\end{lemma}
\begin{proof}
Let $C\subseteq \wt {\mathcal R}$ be any irreducible component, $\wt C$ its normalization,
$j:\,\wt C \to \wt X$ the induced map and $\varphi=J_0\circ j:\,\wt C \to \Alb_0(\wt X)$ the composition.
We have to prove that $\varphi(\wt C)$ is a point.

We argue by contradiction. Assume that $\varphi(\wt C)$ is not a point.
Then the induced map
$$\varphi^*:~H^0\left(\Alb_0(\wt X),\,\Omega_{\Alb_0(\wt X)}^1\right) \lra H^0\left(\wt C,\,\Omega_{\wt C}^1\right)$$
is non-zero. On the other hand, it is clear that $\varphi^*$ factors through
$$H^0\left(\Alb_0(\wt X),\,\Omega_{\Alb_0(\wt X)}^1\right) \overset{J_0^*}\lra H^0\left(\wt X,\,\Omega_{\wt X}^1\right)
\overset{j^*}\lra H^0\left(\wt C,\,\Omega_{\wt C}^1\right).$$
Note that the generator $\tau$ of the Galois group ${\rm Gal}(\wt X/\wt Y)$ acts on $H^0\left(\wt X,\,\Omega_{\wt X}^1\right)$.
Let $$H^0\left(\wt X,\,\Omega_{\wt X}^1\right)_{-1} \oplus H^0\left(\wt X,\,\Omega_{\wt X}^1\right)_{1}$$ be the eigenspace decomposition.
Then by construction, the image of $J_0^*$ is contained in $H^0\left(\wt X,\,\Omega_{\wt X}^1\right)_{-1}$.
To deduce a contradiction, it suffices to prove that the restricted map
$$j^*\big|_{H^0\left(\wt X,\,\Omega_{\wt X}^1\right)_{-1}}:~H^0\left(\wt X,\,\Omega_{\wt X}^1\right)_{-1} \lra H^0\left(\wt C,\,\Omega_{\wt C}^1\right)$$
is zero.

In fact, let $p\in C$ be an arbitrary smooth point of $C$.
Locally around $p$, there exists local coordinate $(x,y)$ such that the action of $\tau$ is given by $\tau(x,y)=(x,-y)$ and $C$ is defined by $y=0$.
For any $1$-form $$\omega=\alpha(x,y)dx+\beta(x,y)dy \in H^0\left(\wt X,\,\Omega_{\wt X}^1\right),$$ one has
$$\omega\in H^0\left(\wt X,\,\Omega_{\wt X}^1\right)_{-1} \Longleftrightarrow \alpha(x,y)=y\tilde \alpha(x,y^2), ~\beta(x,y)=\tilde \beta(x, y^2).$$
Hence if $\omega \in H^0\left(\wt X,\,\Omega_{\wt X}^1\right)_{-1}$, one gets that $j^*\omega\big|_{j^{-1}(p)}=0$,
from which it follows that $j^*\omega=0$ since $p$ is arbitrary.
The proof is complete.
\end{proof}

\begin{lemma}\label{lemma-order-decrease}
Let $y_j \in R_{j}\subseteq Y_{j}$ be a singularity infinitely closed to $y_{i}\in R_{i}\subseteq Y_{i}$ as in the canonical resolution in Figure {\rm\ref{figure-1}}.
Then
\begin{equation*}
m_j\leq m_i,\quad\text{if $m_i$ is even;}\qquad\quad
m_j\leq m_i+1,\quad\text{if $m_i$ is odd.}
\end{equation*}
\end{lemma}
\begin{proof}
It suffices to consider the case where $j=i+1$ and $\psi_{i+1}(y_{i+1})=y_i$.
But this is clear because if $m_i$ is even, then $\mathcal E_{i+1}\nsubseteq R_{i+1}$;
and if $m_i$ is odd, then $\mathcal E_{i+1}\subseteq R_{i+1}$.
\end{proof}

\begin{proof}[Proof of \autoref{prop-restriction-invariants}]
Recall that those blowing-ups $\psi_i$'s are contained in the canonical resolution $\psi$.
For convenience, we view $\psi_i\circ\psi_{i+1}:\,Y_{i+1}\to Y_{i-1}$ as a single blowing-up (but with two exceptional curves)
if $$Y_{i+1}\overset{\psi_{i+1}}\lra Y_{i} \overset{\psi_{i}}\lra Y_{i-1}$$ are blowing-ups of a type-$(2k+1 \to 2k+1)$ singularity.
For a blowing-up $\psi'$ contained in $\psi$,
the order of $\psi'$ is defined to be $k+1$ if $\psi'$ is a blowing-up of a type-$(2k+1 \to 2k+1)$ singularity,
and to be $[m'/2]$ if $\psi'$ is a blowing-up of a singularity of the branch divisor with multiplicity $m'$.
Now we introduce a partial order on these blowing-ups contained in $\psi$:
we say $\psi'\geq \psi''$ if $k'\geq k''$, where $k'$ (resp. $k''$) is the order of $\psi'$ (resp. $\psi''$).
According to \autoref{lemma-order-decrease}, we can reorder these blowing-ups contained in $\psi$ such that
$\psi_i\geq \psi_j$ if $i<j$. Let $M$ be the maximal order of these blowing-ups contained in $\psi$.
Then $\psi$ can be decomposed as
$$\xymatrix{\wt Y\ar@{=}[r] &\hat Y_{M}\ar@/_7mm/"1,5"^-{\psi} \ar[r]^-{\hat\psi_{M}} & \cdots\cdots \ar[r]^-{\hat\psi_2} & \hat Y_1 \ar[r]^-{\hat\psi_1} & \hat Y_0 \ar@{=}[r]& Y}
$$
such that the order of each blowing-up contained in $\hat \psi_i$ is $M+1-i$.

Consider any blowing-up $\psi'$ contained in $\hat\psi_i$.
If it is a blowing-up of a type-$\big(2(M-i)+1 \to 2(M-i)+1\big)$ singularity, let $\mathcal E_1$ and $\mathcal E_2$ be the two exceptional curves.
By construction, one of them, saying $\mathcal E_1$ is contained in the branch divisor, hence its strict inverse image on $\wt X$ is a rational curve;
another one, saying $\mathcal E_2$, is not contained in the branch divisor and intersects the branch divisor at most $2\big(M-i\big)+2$ points,
hence the geometric genus of its strict inverse image on $\wt X$ is at most $M-i$ by Hurwitz formula (cf. \cite[\S\,IV.2]{hartshorne-77}).
If $\psi'$ is an ordinary blowing-up with one exceptional curve $\mathcal E$, then one can prove similarly that
the geometric genus of its strict inverse image on $\wt X$ is also at most $M-i$.
In any case, we obtain that the strict inverse image of any exceptional curve of $\hat \psi_i$ has geometric genus at most $M-i$.

Consider first the case when $J_0(\wt X)$ is a curve of geometric genus $g'>0$. In this case, any curve of geometric genus less than $g'$ is contracted by $J_0$.
Hence combining this with the above arguments and \autoref{lemmawtR-contracted}, we conclude that the total inverse image of $\hat R_{M-g'}$ in $\wt X$ is contracted by $J_0$,
where $\hat R_{M-g'}\subseteq \hat Y_{M-g'}$ is the image of $\wt R$.
In particular, the total inverse image of $\hat R_{M-g'}$ is semi-negative definite,
which implies that $\hat R_{M-g'}$ is also semi-negative definite.
By construction, each blowing-up contained in
$$\hat \psi_{M-g'+1}\circ\cdots\circ\hat\psi_{M}:~\wt Y=\hat Y_M \lra \hat Y_{M-g'}$$
has order less than or equal to $g'$.
Thus there exist $n_2+\sum\limits_{k\geq g'} s_{2k+1}$ vertical isolated $(-2)$-curves contained in $\hat R_{M-g'}$ by \autoref{numberofcontraction},
since the image of any isolated $(-2)$-curve contained in $\wt R$ is still an isolated $(-2)$-curve contained in $\hat R_{M-g'}$. Therefore
\begin{equation}\label{eqnpfq_pi>g'1}
\hat R_{M-g'}^2 \leq -2\left(n_2+\sum\limits_{k\geq g'} s_{2k+1}\right).
\end{equation}
By construction, we have
$$\begin{aligned}
\hat R_{M-g'}^2=&\,R^2-\left(\sum_{k\geq g'} 4(2k^2+2k+1)s_{2k+1}+\sum_{k\geq g'+1}4k^2s_{2k}\right)\\
=&\,\hat x\cdot\frac{\omega_h^2}{\gamma-1}+\hat y\big(T+2(\gamma-1)n_2\big)+\hat z\left(\omega_{h}+R\right)\cdot R\\
&~-\left(\sum_{k\geq g'} 4(2k^2+2k+1)s_{2k+1}+\sum_{k\geq g'+1}4k^2s_{2k}\right),
\end{aligned}$$
where
$$\hat x=\frac{-(g+1-2\gamma)^2}{(2g+1-3\gamma)},\qquad \hat y=\frac{-1}{(2g+1-3\gamma)},\qquad \hat z=\frac{2g+2-4\gamma}{2g+1-3\gamma}.
$$
Hence \eqref{eqnq_pi>g'} follows from the above equation together with \eqref{eqn(omegah+R)R} and \eqref{eqnpfq_pi>g'1}.

Finally, let's consider the case when $q_{\pi}>0$. In this case, $J_0(\wt X)$ is of positive dimension since $J_0(\wt X)$ generates $\Alb_0(\wt X)$ by construction,
and any rational curve in $\wt X$ is contracted by $J_0$.
Hence similarly as above, one sees that $\hat R_{M-1}$ is semi-negative definite and
\begin{equation}\label{eqnpfq_pi>01}
\hat R_{M-1}^2 \leq -2\left(n_2+\sum\limits_{k\geq 1} s_{2k+1}\right).
\end{equation}
Therefore, \eqref{eqnq_pi>0} follows from a similar argument as above.
\end{proof}

In order to use \autoref{prop-restriction-invariants}\,(ii),
we have to know when $J_0(\wt X)$ is a curve, where $J_0$ is defined in \eqref{eqndef-of-J_0}.

\begin{lemma}[\cite{cai-98}]\label{lem-cai}
	If $q_{\pi}>\gamma+1$, then the image $J_0(\wt X)\subseteq \Alb_0(\wt X)$ is a curve of genus at least $q_{\pi}$.
\end{lemma}
\begin{proof}
	First note that if $J_0(\wt X)\subseteq \Alb_0(\wt X)$ is a curve, then its genus is at least $q_{\pi}$
	since $J_0(\wt X)$ generates $\Alb_0(\wt X)$ and $\dim \Alb_0(\wt X)=q_{\pi}$.
	Hence it suffices to prove that $J_0(\wt X)$ is a curve.
	
	Let $\wt F$ be a general fibre of $\tilde f$, and $\wt \Gamma=\tilde\pi(\wt F) \subseteq \wt Y$.
	Consider the linear map $$\varsigma:\,\wedge^2 H^{1,0}\big(\Alb_0(\wt X)\big) \cong H^{2,0}\big(\Alb_0(\wt X)\big) \to H^{1,0}(\wt F)$$
	obtained by composing the linear map $$H^{2,0}\big(\Alb_0(\wt X)\big) \lra H^{2,0}(\wt X)$$ with the restriction map
	$$H^{2,0}(\wt X) \cong H^{0}\big(\wt S,\,\omega_{\wt S}\big) \lra H^{0}\big(\wt F,\,\omega_{\wt F}\big) \cong H^{1,0}(\wt F),$$
	where $\omega_{\wt X}$ (resp. $\omega_{\wt F}$) is the canonical sheaf of $\wt X$ (resp. $\wt F$).
	Note that the generator $\tau$ of the Galois group ${\rm Gal}(\wt X/\oly)$ acts on $H^{1,0}\big(\Alb_0(\wt X)\big)$ by multiplying $-1$,
	from which it follows that the image $\im(\varsigma)$ is contained in the invariant subspace
	$H^{0}\big(\wt F,\,\omega_{\wt F}\big)^{\tau} \cong H^{0}\big(\wt C,\,\omega_{\wt C}\big)$.
	In particular, one has
	$$\dim \im(\varsigma) \leq \dim H^{0}\big(\wt C,\,\omega_{\wt C}\big)=\gamma.$$
	On the other hand, if $J_0(\wt X)$ is a surface, then it is proved by Xiao (cf. \cite[Theorem\,2]{xiao-87},
	see also \cite[Lemma\,1]{pirola-89} by Pirola) that
	$$\dim \im(\varsigma) \geq q_{\pi} -1.$$
	From the two above inequalities it follows that $J_0(\wt X)$ is a curve if $q_{\pi}>\gamma+1$.
\end{proof}

\subsection{Slope of double cover fibrations}\label{sec-slope-double}
In this subsection, we would like to consider the question on the lower bound of the slope for double cover fibrations.
The main techniques are \autoref{thminvariants-double-fibration} and \autoref{prop-restriction-invariants}.

Based on \autoref{thminvariants-double-fibration}, we can reprove the following lower bound of the slope for a double cover fibration,
which was proved earlier by Barja, Zucconi, Cornalba and Stoppino.
\begin{theorem}[{\cite[Cor.\,2.6]{barja-zucconi-01} \& \cite[Thm.\,2.1]{barja-01} \& \cite[Thm.\,3.1\,\&\,3.2]{cornalba-stoppino-08}}]\label{thm-double-1}
Let $f$ be a double cover fibration of type $(g,\gamma)$. If $h$ is locally trivial or $g\geq 4\gamma+1$, then
\begin{equation}\label{eqn4(g-1)/(g-gamma)}
\lambda_f\geq \frac{4(g-1)}{g-\gamma}.
\end{equation}
\end{theorem}
\begin{proof}
By \autoref{thminvariants-double-fibration}, for any $\lambda$, one has
\begin{eqnarray}
&&~(2g+1-3\gamma)(\omega_f^2-\lambda\cdot\chi_f)\label{eqnomega_f^2-lambda-chi_f}\\[0.15cm]
&=\hspace{-0.2cm}&\left(\frac{(3g+1-4\gamma)(g-1)}{2}-\frac{(g+1-2\gamma)^2\lambda}{8}\right)\cdot\frac{\omega_h^2}{\gamma-1}-2(2g+1-3\gamma)\lambda\cdot\chi_h\qquad\nonumber\\
&&\hspace{-0.2cm}+\frac{12-\lambda}{8}\cdot T+\frac{4(g-1)-(g-\gamma)\lambda}{4}\cdot s_2+\frac{(2g+1-3\gamma)\lambda}{4}\cdot n_2\nonumber\\
&&\hspace{-0.2cm}+\sum_{k\geq 1} \Big((12-\lambda)k\big((g-1)+(k-1)(\gamma-1)\big)-(2g+1-3\gamma)\Big)\cdot s_{2k+1}\nonumber\\
&&\hspace{-0.2cm}+\sum_{k\geq 2}\left(\frac{(12-\lambda)k\big((g-1)+(k-2)(\gamma-1)\big)}{2}-2(2g+1-3\gamma)\right)\cdot s_{2k}.\nonumber
\end{eqnarray}
Taking $\lambda=\frac{4(g-1)}{g-\gamma}$ in \eqref{eqnomega_f^2-lambda-chi_f}, it is easy to see that the coefficients of $n_2$ and $s_j$'s for $j\geq 3$  are all non-negative due to \eqref{eqnhurwitz-formula}.
Since $T$, $n_2$ and $s_j$'s for $j\geq 3$ are also all non-negative by definition, it follows from \eqref{eqnomega_f^2-lambda-chi_f} that
\begin{equation}\label{eqnpf-4(g-1)/(g-gamma)-1}
\omega_f^2-\frac{4(g-1)}{g-\gamma}\cdot\chi_f\geq \frac{1}{2(g-\gamma)}\left((g-1)^2\cdot\frac{\omega_h^2}{\gamma-1}+T-16(g-1)\cdot\chi_h\right).
\end{equation}

If $h$ is locally trivial, then $\frac{\omega_h^2}{\gamma-1}=\chi_h=0$ and $T\geq 0$, from which together with \eqref{eqnpf-4(g-1)/(g-gamma)-1}
the inequality \eqref{eqn4(g-1)/(g-gamma)} follows immediately.

If $g\geq 4\gamma+1$ and $\gamma=1$, then by \cite[\S\,V-Theorem\,12.1]{bhpv-04}, one has
\begin{equation}\label{eqnpf-4(g-1)/(g-gamma)-2}
\omega_h \sim_{\text{(numerically equivalent)}} \left(\chi_h+\sum_{i=1}^{n}\frac{l_i-1}{l_i}\right)\Gamma,
\end{equation}
where $\Gamma$ is a general fiber of $h$ and $\{\Gamma_i\}_{i=1,\cdots,n}$ are the union of multiple fibers of $h$ with multiplicities $\{l_i\}_{i=1,\cdots,n}$.
Hence $T=2(g-1)\omega_h\cdot R\geq 4(g-1)^2\chi_h$. Therefore, it follows from \eqref{eqnpf-4(g-1)/(g-gamma)-1} that
$\omega_f^2-4\chi_f\geq 2(g-5)\chi_h\geq 0$.

If $g\geq 4\gamma+1$ and $\gamma>1$, then one has $\omega_h^2\geq \frac{4(\gamma-1)}{\gamma}\cdot\chi_h\geq 0$ and $T\geq 0$.
Hence by \eqref{eqnpf-4(g-1)/(g-gamma)-1}, we get
\[\omega_f^2-\frac{4(g-1)}{g-\gamma}\cdot\chi_f\geq \frac{4(g-1)(g-4\gamma-1)}{2(g-\gamma)\gamma}\cdot\chi_h\geq 0 \,\text{~\,as required.}\qedhere\]
\end{proof}


When $f$ is an irregular double cover, we have the following better bounds, which is a generalization
of \cite[Theorem\,1.4]{lu-zuo-13}.
\begin{theorem}\label{thm-irreg-doub}
Let $f$ be an irregular double cover fibration of type $(g,\gamma)$, and
\begin{equation}
F(g,\gamma,\ell)=(g-1)^2-4(g-1)(\gamma \ell+\gamma+\ell)-4\ell^2(\gamma^2-1).
\end{equation}

{\rm(i)} If $h$ is locally trivial or $F(g,\gamma,1)\geq 0$, then
\begin{equation}\label{eqn6+4(gamma-1)/(g-1)}
\lambda_f\geq 6+\frac{4(\gamma-1)}{g-1}.
\end{equation}

{\rm(ii)} Assume moreover that $J_0(\wt X)$ is a curve, where $J_0$ is defined in \eqref{eqndef-of-J_0}.
If $h$ is locally trivial or $F(g,\gamma,q_{\pi})\geq 0$, then
\begin{equation}\label{eqn-lambda_f>lambda_g,q_pi}
\lambda_f\geq \lambda_{g,\gamma,q_{\pi}}:=8-\frac{4(g+1-2\gamma)}{(q_{\pi}+1)\big((g-1)+(q_{\pi}-1)(\gamma-1)\big)}.
\end{equation}
\end{theorem}
\begin{proof}
We only prove (ii) here, for the proof of (i) is completely the same except replacing the usage of \eqref{eqnq_pi>g'} by \eqref{eqnq_pi>0} in the following.

Note that $J_0(\wt X)$ generates $\Alb_0(\wt X)$ by construction.
Hence the geometric genus of $J_0(\wt X)$ is at least $q_{\pi}=\dim \Alb_0(\wt X)$.
Note also that $\lambda_{g,\gamma,q_{\pi}}\geq \frac{4(g-1)}{g-\gamma}$, since $g+1-2\gamma\geq 0$ by \eqref{eqnhurwitz-formula}.
Hence by \eqref{eqnq_pi>g'} and \eqref{eqnomega_f^2-lambda-chi_f} with $\lambda=\lambda_{g,\gamma,q_{\pi}}$, we obtain
\begin{eqnarray}
\hspace{-0.2cm}&&\omega_f^2-\lambda_{g,\gamma,q_{\pi}}\cdot\chi_f\label{eqn-pf-lambda_f>lambda_g,q_pi-1}\\
\hspace{-0.2cm}&\geq\hspace{-0.1cm}&\frac{8(g-1)-(g+1-2\gamma)\lambda_{g,\gamma,q_{\pi}}}{8}\cdot \frac{\omega_h^2}{\gamma-1}-2\lambda_{g,\gamma,q_{\pi}}\cdot\chi_h+\frac{8-\lambda_{g,\gamma,q_{\pi}}}{8(g+1-2\gamma)}\cdot T\quad\nonumber\\
\hspace{-0.2cm}&&\hspace{-0.2cm}+\frac{\lambda_{g,\gamma,q_{\pi}}}{4}\cdot n_2+\sum_{k=1}^{q_{\pi}-1} \xi_k\cdot s_{2k+1}+\sum_{k=2}^{q_{\pi}}\eta_k\cdot s_{2k}
                 +\sum_{k\geq q_{\pi}} \mu_k\cdot s_{2k+1}+\sum_{k\geq q_{\pi}+1}\nu_k\cdot s_{2k},\nonumber
\end{eqnarray}
where
$$\begin{aligned}
\xi_k&\,=\,k^2\lambda_{g,\gamma,q_{\pi}}-(2k-1)^2,\\
\eta_k&\,=\,\frac{(k-1)\big(k\lambda_{g,\gamma,q_{\pi}}-4(k-1)\big)}{2},\\
\mu_k&\,=\,\frac{\big(4k(g-1)+(2k-1)^2(\gamma-1)\big)(8-\lambda_{g,\gamma,q_{\pi}})-(g+1-2\gamma)\lambda_{g,\gamma,q_{\pi}}}{4(g+1-2\gamma)},\\
\nu_k&\,=\,\frac{k\big((g-1)+(k-2)(\gamma-1)\big)(8-\lambda_{g,\gamma,q_{\pi}})-4(g+1-2\gamma)}{2(g+1-2\gamma)}.
\end{aligned}$$
It is easy to see that $\xi_k\geq 0$ for any $1\leq k\leq q_{\pi}-1$, $\eta_k\geq 0$ for any $2\leq k\leq q_{\pi}$, and
$$\begin{aligned}
\mu_k&\geq \mu_{q_{\pi}}=\frac{2(q_{\pi}-1)}{q_{\pi}+1}+\frac{g-\gamma}{(q_{\pi}+1)\big((g-1)+(q_{\pi}-1)(\gamma-1)\big)}\geq0,&~&\forall~k\geq q_{\pi},\\
\nu_k&\geq \nu_{q_{\pi}+1}=0,&&\forall~k\geq q_{\pi}+1.
\end{aligned}$$
Hence by \eqref{eqn-pf-lambda_f>lambda_g,q_pi-1}, one has
\begin{eqnarray}
\hspace{-0.2cm}&&\omega_f^2-\lambda_{g,\gamma,q_{\pi}}\cdot\chi_f\label{eqn-pf-lambda_f>lambda_g,q_pi-2}\\
\hspace{-0.2cm}&\geq\hspace{-0.1cm}&\frac{8(g-1)-(g+1-2\gamma)\lambda_{g,\gamma,q_{\pi}}}{8}\cdot \frac{\omega_h^2}{\gamma-1}-2\lambda_{g,\gamma,q_{\pi}}\cdot\chi_h+\frac{8-\lambda_{g,\gamma,q_{\pi}}}{8(g+1-2\gamma)}\cdot T\quad\nonumber
\end{eqnarray}

If $h$ is locally trivial, then $\frac{\omega_h^2}{\gamma-1}=\chi_h=0$ and $T\geq 0$. Hence \eqref{eqn-lambda_f>lambda_g,q_pi} is clearly true.

If $F(g,\gamma,q_{\pi})\geq 0$ and $\gamma=1$, then by \eqref{eqnpf-4(g-1)/(g-gamma)-2} one has $T=2(g-1)\omega_h\cdot R\geq 4(g-1)^2\chi_h$.
Hence it follows from \eqref{eqn-pf-lambda_f>lambda_g,q_pi-2} that
$$\omega_f^2-\lambda_{g,1,q_{\pi}}\cdot\chi_f\geq \frac{2(g-8q_{\pi}-5)}{q_{\pi}+1}\cdot \chi_h.$$
Note that the assumption $F(g,\gamma,q_{\pi})\geq 0$ implies that $g\geq 8q_{\pi}+5$ when $\gamma=1$.
Thus the above inequality implies that \eqref{eqn-lambda_f>lambda_g,q_pi} holds if $\gamma=1$.

Finally, we consider the case when $F(g,\gamma,q_{\pi})\geq 0$ and $\gamma>1$.
In this case one has $\omega_h^2\geq \frac{4(\gamma-1)}{\gamma}\cdot\chi_h\geq 0$ and $T\geq 0$.
Hence by \eqref{eqn-pf-lambda_f>lambda_g,q_pi-2}, we get
\[\omega_f^2-\lambda_{g,\gamma,q_{\pi}}\cdot\chi_f\geq \frac{2F(g,\gamma,q_{\pi})}{\gamma(q_{\pi}+1)\big((g-1)+(q_{\pi}-1)(\gamma-1)\big)}\cdot\chi_h \geq 0.\qedhere\]
\end{proof}

\begin{remark}
Let $f$ be an irregular double cover fibration of type $(g,\gamma)$.
Similar to the above proof, one can show that
\begin{equation}\label{eqn-slope>6}
\lambda_f \geq 6, \qquad \text{if~} g\geq 6\gamma+7.
\end{equation}
In fact, by \eqref{eqnq_pi>0} with \eqref{eqnomega_f^2-lambda-chi_f}, one obtains that
$$\begin{aligned}
\omega_f^2-6\chi_f &\,\geq \frac{8(g-1)-6(g+1-2\gamma)}{8}\cdot \frac{\omega_h^2}{\gamma-1}-12\chi_h+\frac{1}{4(g+1-2\gamma)}\cdot T\\
&\,\geq
\left\{\begin{aligned}
&-12\chi_h+\frac{1}{4(g-1)}\cdot 4(g-1)^2\chi_h\geq 0,&&\text{if~}\gamma=1,\\[1.5mm]
&\frac{8(g-1)-6(g+1-2\gamma)}{8}\cdot 4\chi_h-12\chi_h\geq 0,&\quad&\text{if~}\gamma\geq 2.
\end{aligned}\right.
\end{aligned}$$
\end{remark}

We end this section with the following lower bound on the slope of double cover fibrations of type $(g,\gamma)$
with $g$ being not big.
It can be viewed as a supplement to \autoref{thm-double-1}.
\begin{theorem}
Let $f$ be a double cover fibration of type $(g, \gamma)$. If $g\leq 4\gamma+1$ and $(g+1-2\gamma)^2\geq 2(2g+1-3\gamma)$, then
\begin{equation}\label{eqn-4.3-2}
\lambda_f\geq \frac{4(g-1)(3g+1-4\gamma)}{(g+1-2\gamma)^2+4\gamma(2g+1-3\gamma)}.
\end{equation}
\end{theorem}
\begin{proof}
Let $\lambda_0:=\frac{4(g-1)(3g+1-4\gamma)}{(g+1-2\gamma)^2+4\gamma(2g+1-3\gamma)}$.
Then $4\leq \lambda_0 \leq \frac{4(g-1)}{g-\gamma}$ by assumptions.

If $\gamma=1$, then the assumptions imply that $\lambda_0=4$ and $g=5$.
Hence \eqref{eqn-4.3-2} follows from \eqref{eqn4(g-1)/(g-gamma)}.
If $\gamma>1$, taking $\lambda=\lambda_0$ in \eqref{eqnomega_f^2-lambda-chi_f} and using \autoref{lemma-4-1}
below to eliminate $s_2$, one obtains
\begin{eqnarray*}
&&\omega_f^2-\lambda_0\cdot\chi_f\\[0.15cm]
&\geq \hspace{-0.2cm}&\left(\frac{(3g+1-4\gamma)(g-1)}{2(2g+1-3\gamma)}-\frac{(g+1-2\gamma)^2\lambda_0}{8(2g+1-3\gamma)}\right)
\cdot\frac{\omega_h^2}{\gamma-1}-2\lambda_0\cdot\chi_h+\frac{(\lambda_0-4)}{8(\gamma-1)}\cdot T\\
&&\hspace{-0.2cm}+\frac{\lambda_0}{4}\cdot n_2
+\sum_{k\geq 1} \big(k^2\lambda_0-(2k-1)^2\big)\cdot s_{2k+1}
+\sum_{k\geq 2}\Big(\frac{k(k-1)}{2}\lambda_0-2(k-1)^2\Big)\cdot s_{2k}\\
&\geq \hspace{-0.2cm}&\left(\frac{(3g+1-4\gamma)(g-1)}{2(2g+1-3\gamma)}-\frac{(g+1-2\gamma)^2\lambda_0}{8(2g+1-3\gamma)}\right)
\cdot\frac{\omega_h^2}{\gamma-1}-2\lambda_0\cdot\chi_h\\
&\geq \hspace{-0.2cm}&\Bigg(\left(\frac{(3g+1-4\gamma)(g-1)}{2(2g+1-3\gamma)}-\frac{(g+1-2\gamma)^2\lambda_0}{8(2g+1-3\gamma)}\right)
\cdot\frac{4}{\gamma}-2\lambda_0\Bigg)\cdot\chi_h=0,
\end{eqnarray*}
where the second inequality follows from the non-negativity of $T,\,n_2$ and $s_j$'s for $j\geq 3$;
and the third inequality comes comes from the slope inequality $\omega_h^2\geq \frac{4(\gamma-1)}{\gamma}\chi_h$ of the fibration $h$.
The proof is complete.
\end{proof}

\begin{lemma}\label{lemma-4-1}
\begin{equation}\label{eqn-4.3-1}
T+(\gamma-1)\left(s_2+\sum_{k\geq1}4k(2k+1)s_{2k+1}+\sum_{k\geq 2}2k(2k-1)s_{2k}\right)\geq 0.
\end{equation}
\end{lemma}
\begin{proof}
We may assume that $\gamma>1$. By \eqref{eqn(omegah+R)R}, the inequality \eqref{eqn-4.3-1} is equivalent to
\begin{equation}\label{eqn-4.3-lin2}
T+(\gamma-1)\big((\omega_h+R)\cdot R+2n_2\big)\geq 0.
\end{equation}

Let $R=\sum\limits_{i=1}^{m}D_i$ be the decomposition into connected components, such that
$$D_i\cdot \Gamma> 0, \quad\forall~1\leq i\leq l;~\qquad\, D_i\cdot \Gamma= 0, \quad\forall~l+1\leq i\leq m,$$
where $\Gamma$ is a general fiber of $h$. We claim that
\begin{equation}\label{eqn-4.3-lin3}
(\omega_h+D_i)\cdot D_i\geq 0, ~\,\forall~1\leq i\leq l;\quad (\omega_h+D_i)\cdot D_i\geq -2, ~\,\forall~l+1\leq i\leq m.
\end{equation}
Indeed, let $\wt D_i=\sum\limits_{j=1}^{k_i}\wt D_{ij} \to D_i$ be the normalization, and $\sum\limits_{j=1}^{l_i}\wt D_{ij}$ be the irreducible components
which are mapped surjectively onto $B$. Then
$$\begin{aligned}
(\omega_h+D_i)\cdot D_i&~= \big(2g(B)-2\big)\Gamma\cdot D_i+(\omega_Y+D_i)\cdot D_i\\
&~\geq \big(2g(B)-2\big)\Gamma\cdot D_i+\sum_{j=1}^{k_i}\big(2g(\wt D_{ij})-2\big)+2(k_i-1)\\
&~\geq \sum_{j=l_i+1}^{k_i}\big(2g(\wt D_{ij})-2\big)+2(k_i-1) \geq 2(k_i-l_i-1).
\end{aligned}$$
Hence \eqref{eqn-4.3-lin3} follows. Let $D=\sum\limits_{i=1}^{l}D_i$ and $D'=\sum\limits_{i=l+1}^{m}D_i$. Then $(\omega_h+D)\cdot D\geq 0$ by \eqref{eqn-4.3-lin3}.
Since $\Gamma\cdot \big((g+1-2\gamma)\omega_h-(\gamma-1)D\big)=0$, one gets by Hodge index theorem that
$$\begin{aligned}
0&~\geq \big((g+1-2\gamma)\omega_h-(\gamma-1)D\big)^2\\
&~=\big((g+1-2\gamma)\omega_h-(\gamma-1)R\big)^2-(\gamma-1)^2(\omega_h+D')\cdot D'\\
&~\quad+(\gamma-1)(2g+1-3\gamma)\omega_h\cdot D'\\
&~\geq \big((g+1-2\gamma)\omega_h-(\gamma-1)R\big)^2-(\gamma-1)^2(\omega_h+D')\cdot D'.
\end{aligned}$$
Combining this with the fact that
$$(\omega_h+R)\cdot R=(\omega_h+D)\cdot D+(\omega_h+D')\cdot D'\geq (\omega_h+D')\cdot D',$$
we obtain \eqref{eqn-4.3-lin2}, and hence complete the proof.
\end{proof}

\subsection{Proof of \autoref{prop-2-21}}\label{sec-pf-prop-2-21}
	By \autoref{thm-double-1}, we may assume that $q_{f}\geq 2\gamma$.
	Together with \cite[Theorem\,3]{xiao-87a}, we may assume that $q_f\geq \max\{3,2\gamma\}$.
	Note that $q_h\leq \gamma$.	In particular $q_{\pi}=q_f-q_h>0$.
	If $\gamma=1$, then it follows clearly that $q_{f}\geq 3>\gamma+1$;
	if $h$ is locally trivial, then by \autoref{thm-irreg-doub}\,(i),
	we may assume that $q_f\geq \frac{2g+4}{3}>\gamma+1$;
	if $\lambda_h> \frac{4(\gamma-1)}{\gamma-q_h/2}$, then by \autoref{thm-double-1} and its proof, one obtains
	$\lambda_f> \frac{4(g-1)}{g-\gamma}$, from which we may also assume that $q_{f}>\gamma+1$.
	Thus we assume that $q_{\pi}>\gamma+1$ in the rest part of the proof.
	
	By \autoref{lem-cai} together with \eqref{eqnq_pi>g'} and \eqref{eqnomega_f^2-lambda-chi_f}
	for $\lambda=\lambda_0=\frac{4(g-1)}{g-q_f/2}$, we obtain
	\begin{eqnarray}
	&&\omega_f^2-\lambda_0\cdot\chi_f\nonumber\\
	&\geq\hspace{-0.1cm}&\frac{8(g-1)-(g+1-2\gamma)\lambda_0}{8}\cdot \frac{\omega_h^2}{\gamma-1}-2\lambda_0\cdot\chi_h+\frac{8-\lambda_0}{8(g+1-2\gamma)}\cdot T\quad\nonumber\\
	&&\hspace{-0.2cm}+\frac{\lambda_0}{4}\cdot n_2+\sum_{k=1}^{q_{\pi}-1} \xi_k\cdot s_{2k+1}+\sum_{k=2}^{q_{\pi}}\eta_k\cdot s_{2k}
	+\sum_{k\geq q_{\pi}} \mu_k\cdot s_{2k+1}+\sum_{k\geq q_{\pi}+1}\nu_k\cdot s_{2k}\nonumber\\
	&\geq\hspace{-0.1cm}&\frac{8(g-1)-(g+1-2\gamma)\lambda_0}{8}\cdot \frac{\omega_h^2}{\gamma-1}-2\lambda_0\cdot\chi_h+\frac{8-\lambda_0}{8(g+1-2\gamma)}\cdot T.\quad\nonumber
	\end{eqnarray}
	where
	$$\begin{aligned}
	\xi_k&\,=\,k^2\lambda_0-(2k-1)^2,\\
	\eta_k&\,=\,\frac{(k-1)\big(k\lambda_0-4(k-1)\big)}{2},\\
	\mu_k&\,=\,\frac{\big(4k(g-1)+(2k-1)^2(\gamma-1)\big)(8-\lambda_0)-(g+1-2\gamma)\lambda_0}{4(g+1-2\gamma)},\\
	\nu_k&\,=\,\frac{k\big((g-1)+(k-2)(\gamma-1)\big)(8-\lambda_0)-4(g+1-2\gamma)}{2(g+1-2\gamma)}.
	\end{aligned}$$
	
	If $h$ is locally trivial, then $\frac{\omega_h^2}{\gamma-1}=\chi_h=0$ and $T\geq 0$.
	Hence $\omega_f^2-\lambda_0\cdot\chi_f$.
	Moreover, if the equality holds, then the above inequality shows that all the invariants $s_i$'s, $n_2$ and $T$ are vanishing,
	which implies that $\omega_f^2=0$ by \autoref{thminvariants-double-fibration}, contradicting the non-triviality of $f$.
	Hence the strict inequality \eqref{eqn-2-22} follows.
	
	Next, we consider the case when $h$ is not locally trivial.
	By \autoref{lem-cai},
	$J_0(\wt X)\subseteq \Alb_0(\wt X)$ is a curve of genus $\gamma'\geq q_{\pi}$ since $q_{\pi}>\gamma+1$.
	Restricting $J_0$ on the general fiber of $f$, one obtains a map
	$$J_0\big|_{F}:\, F \lra J_0(\wt X).$$
	Since $f$ is not locally trivial, $\deg\big(J_0\big|_{F}\big)\geq 2$.
	If $\deg\big(J_0\big|_{F}\big)=2$, then $J_0\times f$ realizes $S$
	as a double cover of the trivial fibration $J_0(\wt X) \times B$;
	namely, $f$ is a double cover fibration whose associated quotient fibration is trivial.
	Hence by the above arguments, \eqref{eqn-2-22} holds.
	Thus $\deg\big(J_0\big|_{F}\big)\geq 3$. In particular, by the Riemann-Roch formula, one has
	$$q_{\pi}\leq \frac{g+2}{3}.$$
If $\gamma=1$, then by \eqref{eqnpf-4(g-1)/(g-gamma)-2} one has $T=2(g-1)\omega_h\cdot R\geq 4(g-1)^2\chi_h$.
Hence
$$\omega_f^2-\lambda_0\cdot\chi_f \geq \bigg(\frac{(8-\lambda_0)(g-1)}{2}-2\lambda_0\bigg)\chi_h>0.$$
If $\gamma\geq 2$, then
$$\begin{aligned}
\omega_f^2-\lambda_0\cdot\chi_f &\,\geq \bigg(\frac{8(g-1)-(g+1-2\gamma)\lambda_0}{2\gamma-q_h}-2\lambda_0\bigg)\chi_h\\
&\,= \frac{4(g-1)(g+q_h-2\gamma-1-q_{\pi})}{2\gamma-q_h}\chi_h>0.
\end{aligned}$$
This completes the proof.
\qed

\subsection{Proof of \autoref{thm-3-4}}\label{sec-pf-thm-3-4}
According to \autoref{thm-chx} and \cite[Theorem\,3]{xiao-87a}, one may assume that $q_f\geq 2$,
which implies that $g\geq 9q_f\geq 18$ by assumption.
\begin{enumerate}
	\item[$\bullet$] If $g\geq 4\gamma+1$, then according to \autoref{thm-double-1} we may assume that $q_f>\gamma$.
	Hence $f$ is an irregular double cover (cf. \autoref{def-irr-double}), and $g\geq 6\gamma+7$
	since $g\geq 9q_f\geq 9(\gamma+1)$. Therefore \eqref{conjectureequ} follows from \eqref{eqn-slope>6}.\vspace{1mm}
	\item[$\bullet$] If $4\gamma+1>g \geq 4\gamma-2$, then \eqref{conjectureequ} follows from \eqref{eqn-4.3-2}, since in this case
	\[\frac{4(g-1)(3g+1-4\gamma)}{(g+1-2\gamma)^2+4\gamma(2g+1-3\gamma)}>\frac{9(g-1)}{2g}\geq \frac{4(g-1)}{g-q_f}.\]
\end{enumerate}
This completes the proof.
\qed

%

\section{Examples}\label{sec-examples}
In this section, we construct counterexamples with $q_f=\frac{g+1}{2}$ violating Barja-Stoppino's conjecture.
\begin{example}\label{example-2}
We construct a relatively minimal fibration $f:\,X \to E$ of curves of odd genus $g\geq 3$ over an elliptic curve $E$ with $q_f=\frac{g+1}{2}$ and
\begin{equation*}
\lambda_f = 8-\frac{4}{g-1}<8=\frac{4(g-1)}{g-q_f}.
\end{equation*}

Let $E$ be any elliptic curve, and $C$ be any smooth curve of genus $g_0\geq 3$ which admits a double cover to $E$:
$$
\xymatrix{\eta:~C \ar[r]^-{2:1} & E.}$$
Let $\Delta\subseteq C\times C$ be the diagonal, $\sigma$ the involution on $C\times C$ defined by exchanging the two factors, and $X=C\times C/\langle\sigma\rangle$ the quotient surface.
Since $\sigma$ has no isolated fixed point, $X$ is smooth.
According to \cite[\S\,2.4-Example\,(b)]{lopes-pardini-12},
we know that $X$ is minimal of general type with $q(X)=g_0$ and
$$\chi(\calo_X)=\frac{(g_0-1)^2-(g_0-1)}{2},\qquad
\omega_X^2=4(g_0-1)^2-5(g_0-1).$$

To obtain a fibration on $X$, we consider first the fibration on $C\times C$ defined by
$$h:~C\times C \lra E,\qquad (x_1,x_2) \mapsto \eta(x_1)+\eta(x_2),$$
where `$+$' is the addition on the elliptic curve $E$.
It is easy to see that the morphism $h$ factors through $X$ and so induces a fibration $f:\,X \to E$:
$$\xymatrix{
 C\times C\ar[rr]^-{\pi}\ar[dr]_-{h}&&X\ar[dl]^-{f}\\
 &E&
}$$

It is clear that $f$ is relatively minimal since $X$ is minimal, and $q_f=q(X)-g(E)=g_0-1$.
To compute the genus $g$ of a general fiber of $f$, let $H$ be a general fiber of $h$, $F=\pi(H)\subseteq X$, $p=h(H)\in E$,
and $pr_1$ (resp. $pr_2$) be the projection of $C\times C$ to the first (resp. the second) factor $C$.
Then for any $(x_1,x_2)\in H$, one has $\eta(x_1)+\eta(x_2)=p$, i.e.,
$\eta(x_1)=-\eta(x_2)+p$. In other word, one has the following commutative diagram
$$\xymatrix{
 H\ar[rr]^{pr_1|_{H}}\ar[d]_-{pr_2|_{H}}&&C\ar[d]^-{\eta}\\
 C\ar[rr]^-{-\eta+p}&&E
}$$
The maps in the above diagram are all double covers, and the branch divisor of $pr_2|_{H}$ is
$$T=\big\{x\in C ~\big|~y:=-\eta(x)+p\text{~is a branch point of~}\eta:\,C \to E\big\},$$
which is of degree $4g_0-4$.
Hence one obtains that $g(H)=4g_0-3$.
Note that $H\cdot \Delta=8$. Thus by Hurwitz formula, we get that
$$2g(H)-2=2(2g(F)-2)+8.$$
Hence $g=g(F)=2g_0-3$.
Therefore $q_f=g_0-1=\frac{g+1}{2}$, and
$$\lambda_f=\frac{\omega_f^2}{\chi_f}=\frac{\omega_X^2}{\chi(\calo_X)}=\frac{8g_0-18}{g_0-2}=8-\frac{4}{g-1}<8=\frac{4(g-1)}{g-q_f},\text{~as required.}$$
\end{example}

\begin{example}\label{example-3}
We construct a relatively minimal double cover fibration $f:X \to \mathbb P^1$ of type $(g,\gamma)$ with $0<\gamma<(g+1)/2$, $q_f=(g+1)/2$, and
$$\lambda_f=8-\frac{4}{(g+1-2\gamma)\gamma}<8=\frac{4(g-1)}{g-q_f}.$$

Consider the ruled surface $\eta_0:\,\mathbb P^1 \times \mathbb P^1 \to \mathbb P^1.$
Let $\Lambda_0$ be a pencil on $\mathbb P^1 \times \mathbb P^1$ such that $H_0$ is a section of $\eta_0$ and $H_0^2=2$
for a general member $H_0\in \Lambda_0$.
Assume that $\Lambda_0$ has two distinct base-points, which are mapped to $\{p,\,p'\}\subseteq \mathbb P^1$ by $\eta_0$.
Let $\psi:\,\mathbb P^1 \to \mathbb P^1$ be a double cover branched exactly over $\{p,\,p'\}$,
and consider the Cartesian product
$$\xymatrix{\mathbb P^1 \times \mathbb P^1 \ar[rr] \ar[d]_-{\eta} && \mathbb P^1 \times \mathbb P^1 \ar[d]^{\eta_0}\\
\mathbb P^1 \ar[rr]^{\psi} && \mathbb P^1
}$$
Let $\Lambda$ be the pulling-back of $\Lambda_0$.
Then $\Lambda$ also has two distinct base-points ($H$ and $H'$ are tangent to each other at each of these two base-points for any two general $H,H'\in\Lambda$).
Let $\xi:\,\mathbb P^1 \times \mathbb P^1 \to \mathbb P^1$ be another fibration,
and $\{D_1,D_2,\cdots, D_{2\gamma+2}\}$ be $2\gamma+2$ fibers of $\xi$ such that these two base-points of $\Lambda$
are contained in $D_1$ and $D_2$ respectively.
Let $\Gamma \to \mathbb P^1$ be the double cover branched over $\big\{\xi(D_1),\xi(D_2),\cdots, \xi(D_{2\gamma+2})\}$,
and $$Y=\big(\mathbb P^1 \times \mathbb P^1\big) \times_{\mathbb P^1} \Gamma=\mathbb P^1 \times \Gamma$$ the fiber-product.
Let $\Lambda_Y$ be the inverse of $\Lambda$ on $Y$.
Then $\Lambda_Y$ has also exactly two base-points (each of the base-points is of multiplicity two).
Blowing up the base-points of the pencil $\Lambda_Y$, we obtain a fibration
$$\varphi:~ \wt Y \to \mathbb P^1.$$
By construction, the strict inverse images of $D_1$ and $D_2$ in $\wt Y$
are contracted by $\varphi$.
Let $\tilde p,\, \tilde p'$ be the images, and $\Gamma' \to \mathbb P^1$ the double cover branched over $\{\tilde p,\,\tilde p',\,x_1,\cdots,x_{2\gamma'}\}$,
where $\gamma'=(g+1)/2-\gamma$, and $x_1,\cdots,x_{2\gamma'}$ are distinct general points on $\mathbb P^1$.
Let $X$ be the normalization of the fiber-product $\wt Y \times_{\mathbb P^1} \Gamma'$ and $f:\,X \to \mathbb P^1$ the induced fibration as follows
$$\xymatrix{
\Gamma' \ar[d] && X \ar[ll]_-{\phi'} \ar[d]_-{\pi} \ar@/^1mm/"4,4"^-{f} \ar@/^2mm/"2,7"^-{\phi}&& &&\\
\mathbb P^1 && \wt Y \ar[ll]_-{\varphi} \ar[rr] && Y= \mathbb P^1\times \Gamma \ar[d]\ar[rr] \ar@/_5mm/"4,4"_-{h} && \Gamma \ar[d]\\
&&  && \mathbb P^1\times \mathbb P^1 \ar[rr]^-{\xi} \ar[ld]^-{\eta} && \mathbb P^1\\
&&&\mathbb P^1&&&
}$$
Let $\wt C_i=\varphi^*(x_i)$ be the fibers of $\varphi$ for $1\leq i\leq 2\gamma'$.
Then it is clear that
$$\omega_{\wt Y}^2=-8(\gamma-1)-2,\qquad \chi(\mathcal O_{\wt Y})=-(\gamma-1),
\qquad \omega_{\wt Y}\cdot \wt C_i=4\gamma-4.$$
Note that the fibers of $\varphi$ over $\tilde p$ and $\tilde p'$ are of multiplicity two.
Hence $\pi$ is a double cover branched exactly over $\wt R=\big\{\wt C_1,\cdots,\wt C_{2\gamma'}\big\}$.
Therefore, $f$ is a relatively minimal fibration of genus $g$, and
\begin{eqnarray*}
\omega_f^2&=&2\left(\omega_{\wt Y}+\frac12\wt R\right)^2+8(g-1)=8(g+1-2\gamma)\gamma-4,\\
\chi_f&=&2\chi(\mathcal O_{\wt Y})+\frac12\left(\omega_{\wt Y}+\frac12\wt R\right)\cdot\frac{\wt R}{2}+(g-1)=(g+1-2\gamma)\gamma.
\end{eqnarray*}
Hence $f$ has the required slope.
Note that $q(\wt Y)=\gamma$ and $q(X)-q(\wt Y)=\gamma'$
since $\pi$ is the normalization of the fiber-product $\wt Y \times_{\mathbb P^1} \Gamma'$.
Therefore $q_f=\gamma+\gamma'=(g+1)/2$ as required.
\end{example}

{\vspace{2mm}\bf Acknowledgements:}
We are grateful to M. Barja and L. Stoppino for the discussions and many useful suggestions.

\enddocument